\documentclass[12pt]{amsart}
\usepackage[margin=1in,includeheadfoot]{geometry}
\usepackage{amsmath}
\usepackage{slashed}
\usepackage{graphicx}

\usepackage[margin=1in,includeheadfoot]{geometry}
\usepackage{amsmath}
\usepackage{slashed}
\usepackage{graphicx}
\usepackage{mathrsfs}
\usepackage{txfonts}
\usepackage{amssymb,dsfont}
\usepackage{enumerate}
\usepackage{slashed}
\usepackage{fancyhdr}
\usepackage{aliascnt}
\usepackage{pdfsync}
\usepackage{hyperref}

\newtheorem{thm}{Theorem}[section]
\newtheorem{cor}[thm]{Corollary}
\newtheorem{lem}[thm]{Lemma}
\newtheorem{prop}[thm]{Proposition}
\theoremstyle{definition}
\newtheorem{defn}[thm]{Definition}
\theoremstyle{remark}

\numberwithin{equation}{section}

\newcounter{stepnum}

\def\bee{\begin{eqnarray}}
\def\beee{\begin{eqnarray*}}
\def\eee{\end{eqnarray}}
\def\eeee{\end{eqnarray*}}

\def\ba{\begin{array}}
\def\ea{\end{array}}

\def\R{\mathbb R}

\newcommand{\Om}{\Omega}

\newcommand{\ed}{{\rm d}}

\newcommand{\dv}{{\rm d}^{\ast}}
\newcommand{\D}{{\nabla}}

\newcommand{\fr}[2]{\frac{#1}{#2}}

\begin{document}

\title[Harmonic maps with free boundary]{The qualitative behavior at the free boundary for approximate harmonic maps from surfaces}

\author[Jost]{J\"urgen Jost}
\address{Max Planck Institute for Mathematics in the Sciences\\ Inselstrasse 22\\ 04103 Leipzig, Germany}
\address{Department of Mathematics\\ Leipzig University\\ 04081 Leipzig, Germany}
\email{jost@mis.mpg.de}

\author[Liu]{Lei Liu}%
\address{Department of mathematics\\Tsinghua University \\ HaiDian road\\ BeiJing, 100084 \\China}%
\address{Max Planck Institute for Mathematics in the Sciences\\ Inselstrasse 22 \\  04103 Leipzig,  Germany }
\email{leiliu@mis.mpg.de or llei1988@mail.ustc.edu.cn}%

\author[Zhu]{Miaomiao Zhu}
\address{School of Mathematical Sciences, Shanghai Jiao Tong University\\ 800 Dongchuan Road \\ Shanghai, 200240 \\China}
\email{mizhu@sjtu.edu.cn}

\thanks{The research leading to these results has received funding from the European Research
Council under the European Union's Seventh Framework Program
(FP7/2007-2013) / ERC grant agreement no. 267087. Miaomiao Zhu was supported in part by National Natural Science Foundation of China (No. 11601325). We would like to thank the referee for careful
comments and useful suggestions in improving the presentation of the paper}


\date{\today}
\begin{abstract}
Let $\{u_n\}$ be  a sequence of maps from a compact Riemann surface $M$ with smooth boundary to a general compact Riemannian manifold $N$ with free boundary on   a smooth submanifold $K\subset N$ satisfying
\[
\sup_n \ \left(\|\nabla u_n\|_{L^2(M)}+\|\tau(u_n)\|_{L^2(M)}\right)\leq \Lambda,
\]
where  $\tau(u_n)$ is the tension field of the map $u_n$. We show that the energy identity and the no neck property hold during a blow-up process. The assumptions are such that this result also applies to the harmonic map heat flow  with free boundary, to prove the energy identity at finite singular time as well as at infinity time. Also, the no neck property holds at infinity time.
\end{abstract}
\maketitle
\section{introduction}
Let $(M,g)$ be a compact Riemannian manifold with smooth boundary and $(N,h)$ be a compact Riemannian manifold of dimension $n$. Let $K\subset N$ be a $k-$dimensional closed submanifold where $1\leq k\leq n$. For a mapping $u\in C^2(M,N)$, the energy density of $u$ is defined by
\[
e(u)=\frac{1}{2}|\nabla u|^2={\rm Trace}_gu^*h,
\]
where $u^*h$ is the pull-back of the metric tensor $h$.

The energy of the mapping $u$ is defined as $$E(u)=\int_Me(u)dvol_g.$$

Define
\[
C(K)=\left \{u\in C^2(M,N);u(\partial M)\subset K \right \}.
\]
A critical point of the energy $E$ over $C(K)$ is a harmonic map with free boundary $u(\partial M)$ on $K$. The problem of the existence, uniqueness and regularity of such harmonic maps with a free boundary was first systematically investigated in \cite{GJ}.

By Nash's embedding theorem,  $(N,h)$ can be isometrically embedded into some Euclidean space $\mathbb{R}^N$. Then we can get the Euler-Lagrange equation
\[
\Delta_g u=A(u)(\nabla u,\nabla u),
\]
where $A$ is the second fundamental form of $N\subset \mathbb{R}^N$ and $\Delta_g$ is the Laplace-Beltrami operator on $M$ which is defined by $$\Delta_g:=-\frac{1}{\sqrt{g}}\frac{\partial}{\partial x^\beta}(\sqrt{g}g^{\alpha\beta}\frac{\partial}{\partial x^\alpha}).$$ Moreover, for $1\leq k\leq n-1$, $u$ has free boundary $u(\partial M)$ on $K$, that is
\begin{align}\label{FB}
u(x)\in K,\quad du(x)(\overrightarrow{n})\perp T_{u(x)}K, \quad a.e.\ x\in\partial M,
\end{align}
where $\overrightarrow{n}$ is the outward unite normal vector on $\partial M$ and $\perp$ means orthogonal.

Specially, for $k=n$, $u$ satisfies a homogeneous Neumann condition on $K$, that is
\begin{align}
u(x)\in K,\quad du(x)(\overrightarrow{n})=0, \quad a.e.\ x\in\partial M.
\end{align}

The tension field $\tau(u)$ is defined by
\begin{align}\label{HM}
&\tau(u)=-\Delta_g u+A(u)(\nabla u,\nabla u).
\end{align}
Thus, $u$ is a harmonic map if and only if $\tau(u)=0$.

When we consider a limit of a sequence of maps with uniformly $L^2$-bounded tension fields, the domain may decompose into several pieces (a phenomenon called bubbling or blow-up), and the limit map satisfies the equations or bounds on each piece. The question is whether the sum of the energies of the limit map on those pieces equals the limit of the energies of the approximating maps. Affirmative results are called energy identity and no neck property, and the approach is called blow-up theory; the precise definitions will be given below. Because the problem is conformally invariant only in dimension 2, the analysis usually needs to be restricted to that case, and this will also apply to this paper.

When $M$ is a closed surface, the compactness problem and the blow-up theory  (energy identity and no neck property) for a sequence of maps $\{u_n\}$ from $M$ to $N$ with uniformly $L^2$-bounded tension fields $\tau(u_n)$ and uniformly bounded energy has been extensively studied (see e.g.
\cite{Jost1991, parker1996bubble, qing1995, DingWeiyueandTiangang,Wang, qing1997bubbling}), since the fundamental work of Sacks-Uhlenbeck \cite{SU}. For sequences of general bounded tension fields, see \cite{LiJiayuandZhuXiangrong,LiJiayuandZhuXiangrong-2, Luoyong,Wang-Wei-Zhang}. For sequences of solutions of more general elliptic systems with an antisymmetric structure, we refer to \cite{LR, LS}. For corresponding results about harmonic map flows, see e.g. \cite{Struwe-1, qing1995, qing1997bubbling, Lin-Wang, Topping-1}. For results of other types of approximate sequences
for harmonic maps, see e.g. \cite{Jost1991, CM, La, Li-Wang2, hy}. For the energy identity of harmonic maps from higher dimensional domains, see \cite{LinR}.

In this paper, we shall study the blow-up analysis for a sequence of maps $\{u_n\}$ from a compact Riemann surface $M$ with smooth boundary $\partial M$ to a compact Riemannian manifold $N$ with uniformly $L^2$-bounded tension fields $\tau(u_n)$, uniformly bounded energy and with free boundary $u_n(\partial M)$ on $K$.
Since the interior case is already well understood, we shall focus on the case where the energy concentration occurs at the free boundary and complete the blow-up theory at the free boundary for a bubbling sequence. When boundary blow-up occurs, the corresponding neck domains are in general not simply half annuli and hence a finer decomposition of the neck domains would be necessary in order to carry out the neck analysis (see Section \ref{sec:energy-identity}).

In fact, we shall first address the regularity problem at the free boundary for weak solutions (see Section \ref{sec:regularity-results}) of
\begin{align}\label{weak-solution}
-\Delta_g u+A(u)(\nabla u,\nabla u)=F\ \ in\ M
\end{align}
for some $F\in L^p(M)$, $p>1$ and under the free boundary constraint \eqref{FB}, as it provides some necessary elliptic estimates at the free boundary, which form the analytical foundation of the blow-up theory for the sequence $\{u_n\}$ (see Section \ref{sec:basic-lemma}). We would like to remark that the regularity at the free boundary for weak solutions of \eqref{weak-solution} can be proved by applying the classical reflection methods for the harmonic map case by  Gulliver-Jost \cite{GJ} and Scheven \cite{Scheven} or  a modified reflection method in \cite{CJWZ} and \cite{sharpzhu} which combines H\'{e}lein's moving frame method \cite{helein_conservation} and Scheven's reflection method \cite{Scheven} so that the technique of Rivi\`{e}re-Struwe  in \cite{riviere_struwe} (which holds true also in dimension 2) can be applied. The latter was developed for Dirac-harmonic maps which includes harmonic maps as a special case. In this paper, we shall present an alternative approach without using moving frames (see Section \ref{sec:regularity-results}).

Now, we state our first main result:
\begin{thm}\label{main:thm-1}
Let $u_n:M\to N$ be a sequence of $W^{2,2}$ maps with free boundary $u_n(\partial M)$ on $K$ $(1\leq k\leq n)$, satisfying
\[
E(u_n)+\|\tau(u_n)\|_{L^2(M)}\leq\Lambda<\infty,
\]
where $\tau(u_n)$ is the tension field of $u_n$. We define the blow-up set
 \begin{equation}
\mathcal{S}:=\bigcap_{r>0}\left \{x\in M|\liminf_{n\to\infty}\int_{D^M_r(x)}|du_n|^2dvol\geq\overline{\epsilon}^2\right\},
\end{equation}
where $D^M_r(x)= \left \{y\in M|\ dist(x,y)\leq r \right \}$ denotes the geodesic ball in $M$ and $\overline{\epsilon}>0$ is a constant whose value will be given in  \eqref{def:small-energy}.
Then $\mathcal{S}$ is a finite set  $\{p_1,...,p_I\}$. By taking subsequences, $\{u_n\}$ converges in $W^{2,2}_{loc}(M \setminus \mathcal{S})$ to some limit map $u_0\in W^{2,2}(M,N)$
with free boundary on $K$ and there are finitely many bubbles: a finite set of harmonic spheres $w_i^l:S^2\to N$, $l=1,...,l_i$,
and a finite set of harmonic disks $w_i^k:D_1(0)\to N$,  $k=1,...,k_i$ with free boundaries on $K$, where $l_i,\ k_i\geq 0$ and $l_i+k_i\geq 1$, $i=1,...,I$, such that
\begin{eqnarray}
\lim_{n\to\infty}E(u_n)=E(u_0)+\sum_{i=1}^I\sum_{l=1}^{l_i}E(w^l_i)
+\sum_{i=1}^I\sum_{k=1}^{k_i}E(w^k_i),
\end{eqnarray}
and the image $u_0(M)\cup_{i=1}^I\big(\cup_{l=1}^{l_i}(w^l_i(S^2))
\cup_{k=1}^{k_i}(w^k_i(D_1(0)))\big)$ is a connected set. Here, harmonic spheres are minimal spheres and harmonic disks with free boundary on $K$ are minimal disks with free boundary on $K$.
\end{thm}

In contrast to the Dirichlet problem where, due to the pointwise boundary condition, no blow-up at the boundary is possible. Here, a blow-up may occur at the boundary and produce one or more harmonic disks with the same free boundary $K$ as the original maps. We should also mention that the Plateau boundary condition for minimal surfaces can also be seen as a free boundary condition where the target set $K$ is a Jordan curve. Here, the monotonicity condition and the three-point normalization that are usually imposed prevent a boundary blow-up, however, see \cite{GJ} and the systematic discussion in \cite{Jost1991}.


Our results in the above theorem apply to some classical problems like minimal surfaces in Riemannian manifolds with free boundaries, harmonic functions with free boundary (c.f. \cite{LP}) as well as to pseudo holomorphic curves in sympletic manifolds with totally real boundary conditions and Lagrangian boundary conditions, c.f. \cite{Ye, F, MS, IS, We} and to string theory where the free boundary represents a D-brane, c.f. \cite{Jost2009}.


The reason why we work with a sequence of maps with uniformly $L^2$-bounded tension fields and with free boundary is that we want to apply our results in Theorem \ref{main:thm-1} to the following heat flow for harmonic maps with free boundary:
\bee
\partial_tu(x,t)&=&\tau(u(x))\quad (x,t)\in M\times (0,T);\label{HMF:1}\\
u(\cdot,0)&=&u_0(x)\quad x\in M;\label{HMF:2}\\
u(x,t)&\in& K, \quad a.e. \  x\in\partial M, \quad  \forall \ t\geq 0;\label{HMF:3}\\
du(x)(\overrightarrow{n})&\perp& T_{u(x)}K, \quad \forall (x,t)\in\partial M\times (0,T).\label{HMF:4}
\eee
The existence of a global weak solution of (\ref{HMF:1}-\ref{HMF:4}) with finitely many
singularities was considered by Ma \cite{Ma-li}, following the pioneering works by Struwe \cite{Struwe-1,Struwe-2}.
For higher dimensional cases, we refer to \cite{Struwe-3, Chen-Lin}. For other work on the harmonic map flow with free boundary, see \cite{Li}. For the harmonic map flow with Dirichlet boundary condition, we refer to Chang \cite{Chang}.

Let $u:M\times(0,\infty)\to N$ be a global weak solution to (\ref{HMF:1}-\ref{HMF:4}), which is smooth away from a finite number of singular points $\{(x_i,t_i)\}\subset M\times(0,\infty)$. In this paper, we shall complete the qualitative picture at the singularities of this flow, where bubbles (nontrivial harmonic spheres or nontrivial harmonic disks with free boundary) split off.

At infinite time, we have
\begin{thm}\label{thm:infty-time}
There exist a harmonic map $u_\infty:M\to N$ with free boundary in $K$, a finite number of bubbles $\{\omega_i\}_{i=1}^m$ and sequences $\{x^i_n\}_{i=1}^m\subset M$, $\{\lambda^i_n\}_{i=1}^m\subset \mathbb{R}_+$ and $\{t_n\}\subset \mathbb{R}_+$ such that
\begin{align}
\lim_{t\nearrow \infty}E(u(\cdot,t),M)=E(u_\infty,M)+\sum_{i=1}^mE(\omega_i)
\end{align}
and
\begin{align}\label{equation:30}
\|u(\cdot,t_n)-u_\infty(\cdot)-\sum_{i=1}^m\omega^i_n(\cdot)\|_{L^\infty(M)}\to 0
\end{align}
as $n\to \infty$, where $\omega^i_n(\cdot)=\omega^i\left(\frac{\cdot-x^i_n}{\lambda^i_n}\right)-\omega_i(\infty)$. Here, \eqref{equation:30} is equivalent to say that the image of weak limit $u_\infty$ and bubbles $\{\omega_i\}_{i=1}^m$ is a connected set as in Theorem \ref{main:thm-1}.
\end{thm}

For finite time blow-ups, we have
\begin{thm}\label{thm:finite-time}
For $T_0<\infty$, let $u\in C^\infty(M\times (0,T_0),N)$ be a solution to (\ref{HMF:1}-\ref{HMF:4}) with $T_0$ as its singular time. Then there exist  finite many bubbles $\{\omega_i\}_{i=1}^l$ such that
\begin{align}
\lim_{t\nearrow T_0}E(u(\cdot,t),M)=E(u(\cdot,T_0),M)+\sum_{i=1}^lE(\omega_i).
\end{align}
\end{thm}

To study the regularity or the qualitative behavior at the free boundary for approximate harmonic maps in this paper, we need some new observations. Firstly, we need to extend the solution across the free boundary as in the harmonic map case done by Scheven \cite{Scheven} and the main difficulty is to write the equation of the extended map into an elliptic system with an antisymmetric potential up to some transformation (see Proposition 3.3). Secondly, thanks to the free boundary condition, we can apply the Pohozaev's argument which was firstly introduced by Lin-Wang \cite{Lin-Wang} for approximate harmonic maps, in the local region as $D_r(x)\cap M$ with $x\in \partial M$. See Lemma 4.3. This is crucial when we estimate the energy concentration in the neck domain. Thirdly, we have a finer observation of the neck domain. For the boundary blow-up point, the neck domains consist of some irregular half annulus. We will decompose these irregular neck domains into three parts as: interior parts, regular half annulus with the center points living on the boundary and the remaining parts. The first and third parts are easy to control due to the classical blow-up theory of (approximate) harmonic maps with interior blow-up points. In this paper, we focus on the energy concentration in the domains of the second parts.

Since the extended map satisfies an elliptic system with an antisymmetric potential up to some transformation and with some error term $F$ (see Proposition 3.3), one can utilize the idea in \cite{LR} (with $F=0$) with some modifications to get the energy identity. Here in the present paper, we shall adapt the methods in \cite{DingWeiyue} developed for the interior bubbling case to get the energy identity and the no neck property in the free boundary case. To show the no neck property, namely, bubble tree convergence, we shall get the exponential decay of the energy by deriving a differential inequality on the neck region.

This paper is organized as follows. In Section \ref{sec:prelimilary},  we recall some classical results which will be used in this paper.
In Section \ref{sec:regularity-results}, we derive a new form of the elliptic system for the extended map after involution across the boundary which will allow us to  turn the boundary regularity problem into an interior regularity problem. As a corollary of this boundary regularity result, we prove a removability theorem for singularities at the free boundary. In Section \ref{sec:basic-lemma}, using the new equation of the involuted map, we obtain the small energy regularity in the free boundary case. The gap theorem and Pohozaev's identity in the free boundary case will also be established. In Section \ref{sec:energy-identity}, we prove the energy identity and no neck property at the free boundary by decomposing the neck domain into several parts including a half annulus centered at the boundary and then using the involuted map's equation. Combining this with the interior blow-up theory, we complete the proof of Theorem \ref{main:thm-1}. In Section \ref{set:heat-flow}, we apply Theorem \ref{main:thm-1} to the harmonic map flow with free boundary and prove Theorem \ref{thm:infty-time} and Theorem \ref{thm:finite-time}.

\

\noindent\textbf{Notation:}
$D_r(x_0)$ denotes the closed ball of radius $r$ and center $x_0$ in $\mathbb{R}^2$. Denote
\begin{align*}
 &D^+_r(x_0):=\{x=(x^1,x^2)\in D_r(x_0)|x^2\geq 0\}, D^-_r(x_0):=\{x=(x^1,x^2)\in D_r(x_0)|x^2\leq 0\},\\
  & \partial^+ D_r(x_0):=\{x=(x^1,x^2)\in \partial D_r(x_0)|x^2\geq 0\},
   \partial^- D_r(x_0):=\{x=(x^1,x^2)\in \partial D_r(x_0)|x^2\leq 0\},\\
    &\partial^0 D^+_r(x_0)=\partial^0 D^-_r(x_0):=\partial D^+_r(x_0)\setminus \partial^+ D_r(x_0).
\end{align*}
Let $a\geq 0$ be a constant, denote
\[
\mathbb{R}^2_a:=\{(x^1,x^2)|x^2\geq -a\}\ \ and \ \ \mathbb{R}^{2+}_a:=\{(x^1,x^2)|x^2> -a\}.
\]
For convenience, we denote $D_r=D_r(0)$, $D=D_1(0)$ and $\mathbb{R}^2_+=\mathbb{R}^2_a$ when $a=0$.

Let $T\subset \partial M$ be a smooth boundary portion, denote
\[
W^{k,p}_\partial (T)=\{g\in L^1(T):g=G|_T \mbox{ for some }G\in W ^{k,p}(M)\}
\]
with norm
\[
\|g\|_{W^{k,p}_\partial (T)}=\inf_{G\in W ^{k,p}(M),G|_{T}=g}\|G\|_{W ^{k,p}(M)}.
\]

In this paper, we use the notation $\Delta_g$ (or $\Delta_M$) to denote the Laplace-Beltrami operator on the Riemannian manifold $(M,g)$ and use $\Delta:=\partial^2_x+\partial^2_y$ to denote the usual Laplace operator on $\R^2$.

\

\section{Preliminary results}\label{sec:prelimilary}

\

In this section, we will recall some well known results that are useful for our problem.

Firstly, we recall the interior small energy regularity result (see \cite{DingWeiyueandTiangang,LiJiayuandZhuXiangrong}) which is firstly introduced in \cite{SU}.
\begin{lem}\label{lem:small-energy-regularity-interior}
Let $u\in W^{2,p}(D,N)$ for some $1<p\leq 2$.  There exist constants $\epsilon_1=\epsilon_1(p,N)>0$ and $C=C(p,N)>0$, such that if $\|\nabla u\|_{L^2(D)}\leq\epsilon_1$, then
\begin{equation}
\|u-\frac{1}{\pi}\int_Du(x)dx\|_{W^{2,p}(D_{1/2})}\leq C(p,N)(\|\nabla u\|_{L^p(D)}+\|\tau(u)\|_{L^p(D)}),
\end{equation}
where $\tau(u)$ is the tension field of $u$.

Moreover, by the Sobolev embedding $W^{2,p}(\mathbb{R}^2)\subset C^0(\mathbb{R}^2)$, we have
\begin{equation}
\|u\|_{Osc(D_{1/2})}=\sup_{x,y\in D_{1/2}}|u(x)-u(y)|\leq C(p,N)(\|\nabla u\|_{L^p(D)}+\|\tau(u)\|_{L^p(D)}).
\end{equation}
\end{lem}

Secondly, we recall a gap theorem for the case of a closed domain.
\begin{lem}[\cite{DingWeiyue}]
There exists a constant $\epsilon_0=\epsilon_0(M,N)>0$ such that if $u$ is a smooth harmonic map from a closed Riemann surface $M$
to a compact Riemannian manifold $N$ and satisfying $$\int_M|\nabla u|^2dvol\leq\epsilon_0,$$ then $u$ is a constant map.
\end{lem}

Thirdly, we state an interior removable singularity result.
\begin{thm}[\cite{Li-Wang}]\label{thm:remov-sing-interior}
Let  $u:D\setminus\{0\}\to N$ be a $W^{2,2}_{loc}(D\setminus\{0\})$ map with finite energy that satisfies
\[
\tau(u)=g\in L^2(D,TN),\quad x\in D\setminus\{0\}.
\]
Then $u$ can be extended to a map in $W^{2,2}(D,N)$.
\end{thm}

Next, combining the regularity results for critical elliptical systems with an antisymmetric structure developed by
Rivi{\`e}re  \cite{Ri} and Rivi{\`e}re-Struwe \cite{riviere_struwe} with various extensions in e.g.
\cite{Ri1, Zhu, Sh, Sh_To, schikorra_frames, rupflin, Moser}, we state the following

\begin{thm}\label{thm:regularity-A}
Let $d \geq2$, $0\leq s\leq d$, $0<\Lambda <\infty$ and $1<p<2$. For any $A\in L^{\infty}\cap W^{1,2}(D, GL(d))$,
$\Omega\in L^2(D,so(d)\otimes \wedge^1 \R^m)$, $f\in L^p(D,\R^d)$ and  any $u\in W^{1,2}(D,\R^d)$ weakly solving
\begin{eqnarray}
\dv(A\ed u) &=& \langle \Omega , A \ed u\rangle  + f \quad \text{in}\quad D,  \label{regularity-A}
\end{eqnarray}
with $A$ satisfying
\begin{equation}\label{Abound}
\Lambda^{-1}|\xi| \leq |A(x)\xi| \leq \Lambda |\xi| \,\,\,\,\,\text{for a.e. $x\in D$}, \,\,\,\text{for all $\xi\in \R^d$},
\end{equation}
we have $u\in W^{2,p}_{loc}(D)$ and there exist $\epsilon=\epsilon(d,\Lambda, p)>0$ and $C=C(d,\Lambda, p)>0$ such that whenever
$\|\Om\|_{L^2(D)}+
\|\D A\|_{L^2(D)}\leq \epsilon $ then
$$\|\D^2 u\|_{L^p(D_{\fr12})} + \|\D u\|_{L^{\fr{2p}{2-p}}(D_{\fr12})} \leq C(\| u\|_{L^1(D)}
+ \|f\|_{L^p(D)} ).$$
\end{thm}

It is well known that the harmonic map equation can be written as a critical elliptical system with an antisymmetric structure and hence we have the following (which can also be proved by using classical methods developed for the harmonic map case, see e.g. \cite{helein_conservation})

\begin{thm}\label{thm:regularity-inter.}
For every $p\in(1,\infty)$ there exists an $\epsilon>0$ with the following property. Suppose that $u\in W^{1,2}(D;N)$ and $f\in L^p(D;\mathbb{R}^N)$ satisfy
\[
\tau(u)=f \ \ in \ D
\]
weakly, then $u\in W^{2,p}_{loc}(D)$.
\end{thm}

Finally, we recall the classical boundary estimates for the Laplace operator under Neumann boundary condition.

\begin{lem}[see e.g. \cite{Wehrheim}]\label{lem:lp-estimate-neumann}
Let $f\in W^{k,p}(M)$ and $g\in W_\partial^{k,p}(M)$ for some $k\in\mathbb{N}_0$, $1<p<\infty$. Assume that $u\in W^{1,p}(M)$ weakly solves
\begin{align*}
\Delta_M u=f \quad &in \quad M;\\
\frac{\partial u}{\partial \overrightarrow{n}}=g\quad &on \quad \partial M.
\end{align*}
Then $u\in W^{k+2,p}(M)$ is a strong solution. Moreover, there exist constants $C=C(M)>0$ and $C'=C'(M)>0$ such that for all $u\in W^{k+2,p}(M)$
\begin{align*}
\|u\|_{W^{k+2,p}(M)}&\leq C(\|\Delta_M u\|_{W^{k,p}(M)}+\|\frac{\partial u}{\partial \overrightarrow{n}}\|_{W_\partial^{k+1,p}(M)}+\|u\|_{L^p(M)});\\
\|u\|_{W^{k+2,p}(M)}&\leq C'(\|\Delta_M u\|_{W^{k,p}(M)}+\|\frac{\partial u}{\partial \overrightarrow{n}}\|_{W_\partial^{k+1,p}(M)}), \quad if \quad \int_Mu=0.
\end{align*}
\end{lem}

\

\section{Regularity at the free boundary}\label{sec:regularity-results}

\

In this section, we will prove a regularity theorem for weak solutions of \eqref{weak-solution} and \eqref{FB} with $F\in L^p(M,\R^N)$ for some $p>1$ where $F(x)\in T_{u(x)}N$ for $a.e.\ x\in M$. As an application, we derive the removability theorem for a local singularity at the free boundary.

\

We first need to define weak solutions of \eqref{weak-solution} and \eqref{FB}.
\begin{defn}
$u\in H^1(M,N)$ is called a weak solution to \eqref{weak-solution} and \eqref{FB} if $u(\partial M)\subset K$ a.e. and
\[
-\int_M\nabla u\cdot\nabla \varphi dvol=\int_M F\cdot\varphi dvol
\]
for any vector field $\varphi\in L^\infty\cap H^1(M,TN)$ that is tangential along $u$ and satisfies the boundary condition $\varphi(x)\in T_{u(x)}K$ for a.e. $x\in\partial M$. We also say $u\in H^1(M,N)$ is a weak solution of \eqref{weak-solution} with free boundary $u(\partial M)$ on $K$.
\end{defn}

For a weakly harmonic map with free boundary ($i.e.\ F=0$), it is shown that the image of  the map is contained in a small tubular neighborhood of $K$ if the energy of the map is small,
see Lemma 3.1 in  \cite{Scheven}. The proof there requires the interior $L^{\infty}$-estimate for the gradient of the map. Here, we extend this localization property
to the more general case of weak solutions of \eqref{weak-solution} with $F\in L^p(D^+)$ for some $1<p\leq 2$ and derive certain oscillation estimate for the solution. In our case,  there is in general
no interior $L^{\infty}$-estimate for the gradient of the map.

\begin{lem}\label{lem:coordinate}
Let $F\in L^p(D^+)$ for some $1<p\leq 2$ and $u\in W^{1,2}(D^+,N)$ be a weak solution of \eqref{weak-solution} with
free boundary $u(\partial^0D^+)$ on $K$. Then there exists positive constants $C=C(p,N)$, $\epsilon_2=\epsilon_2(p,N)$, such that if $\|\nabla u\|_{L^2(D^+)}\leq\epsilon_2$, then
\begin{equation}\label{equation:25}
dist(u(x),K)\leq C(p,N)(\|\nabla u\|_{L^2(D^+)}+\|F\|_{L^p(D^+)})\ for\ all\ x\in D_{1/2}^+.\end{equation} Moreover, we have
\begin{equation}\label{equation:26}
Osc_{D^+_{\frac{1}{4}}}u:=\sup_{x,y\in D^+_{\frac{1}{4}}}|u(x)-u(y)|\leq C(p,N)(\|\nabla u\|_{L^2(D^+)}+\|F\|_{L^p(D^+)}).
\end{equation}
\end{lem}
\begin{proof}
We shall follow the scheme of the proof of Lemma 3.1 in \cite{Scheven}. Take $\epsilon_2=\min\{\epsilon_1,\epsilon\}$ where $\epsilon_1$ and $\epsilon$ are the corresponding constants in Lemma \ref{lem:small-energy-regularity-interior} and Theorem \ref{thm:regularity-inter.}. By the interior regularity result Theorem \ref{thm:regularity-inter.},
we know $u\in W^{2,p}_{loc}(D^+\setminus \partial D^+)$.
For any $x_0\in D_{1/2}^+\setminus \partial^0D^+$, set $R=\frac{1}{3}dist(x_0,\partial^0 D^+)$ and suppose $x_1\in\partial^0D^+$ is the nearest point
to $x_0$, $i.e.$ $|x_0-x_1|=dist(x_0,\partial^0 D^+)=3R$. Let $G_{x_0}$ be the fundamental solution of the Laplace operator
with singularity at $x_0$ which satisfies
\begin{align*}
|\nabla G_{x_0}|\leq C(n)|x-x_0|^{-1} \mbox{ for all } x\in \mathbb{R}^2.
\end{align*}
Setting $w(x)=u(x)-\overline{u}$ where $\overline{u}:=\frac{1}{|D^+_{5R}(x_1)|}\int_{D^+_{5R}(x_1)}udx$ and
 choosing a cut-off function $\eta\in C_0^\infty(D_{2R}(x_0))$ such that $0\leq\eta\leq1$, $\eta|_{D_R(x_0)}\equiv 1$ and $|\nabla\eta|\leq\frac{C}{R}$,
by Green's representation theorem and integrating by parts, we have
\begin{align}\label{inequality:01}
|w(x_0)|^2&=-\int_{D_{2R}(x_0)}\nabla G_{x_0}(x)\nabla(|w|^2\eta^2)dx\notag\\
&\leq
C\int_{D_{2R}(x_0)}|\nabla G_{x_0}(x)||w\nabla w|\eta^2dx+C\int_{D_{2R}(x_0)\setminus D_{R}(x_0)}|\nabla G_{x_0}(x)||w|^2|\nabla\eta|dx\notag\\
&\leq
C\|w\|_{L^\infty(D_{2R}(x_0))}\int_{D_{2R}(x_0)}|\nabla G_{x_0}(x)||\nabla u|dx+CR^{-2}\int_{D_{2R}(x_0)\setminus D_{R}(x_0)}|w|^2dx\notag\\
&\leq
C\|w\|_{L^\infty(D_{2R}(x_0))}\|\nabla G_{x_0}(x)\|_{L^{\frac{q}{q-1}}(D_{2R}(x_0))}\|\nabla u\|_{L^{q}(D_{2R}(x_0))}+CR^{-2}\int_{D_{2R}(x_0)}|w|^2dx\notag\\
&:=\mathbb{I}+\mathbb{II},
\end{align}
where $2<q=\frac{p}{2-p}<\frac{2p}{2-p}$ if $1<p<2$ and $q=4$ if $p=2$.

According to Lemma \ref{lem:small-energy-regularity-interior}, we have
\begin{align}\label{equation:01}
R^{1-\frac{2}{s}}\|\nabla u\|_{L^{s}(D_{2R}(x_0))}+\|u\|_{Osc(D_{2R}(x_0))}
&\leq C(s,p,N)(\|\nabla u\|_{L^2(D_{3R}(x_0))}+R^{1-\frac{1}{p}}\|F\|_{L^p(D_{3R}(x_0))})\notag\\
&\leq C(s,p,N)(\|\nabla u\|_{L^2(D^+)}+\|F\|_{L^p(D^+)})
\end{align}
for any $2<s<\frac{2p}{2-p}$. Thus, we obtain
\begin{align*}
\mathbb{I}&\leq C(p,N)\frac{\|\nabla u\|_{L^2(D^+)}+\|F\|_{L^p(D^+)}}{R^{1-2/q}}\|w\|_{L^\infty(D_{2R}(x_0))}\|\frac{1}{|x-x_0|}\|_{L^{\frac{q}{q-1}}(D_{2R}(x_0))}\\
&\leq C(p,N)(\|\nabla u\|_{L^2(D^+)}+\|F\|_{L^p(D^+)})\|w\|_{L^\infty(D_{2R}(x_0))}\\
&\leq C(p,N)(\|\nabla u\|_{L^2(D^+)}+\|F\|_{L^p(D^+)})(|w(x_0)|+\|u\|_{Osc(D_{2R}(x_0))})\\
&\leq \frac{1}{2}|w(x_0)|^2+C(p,N)(\|\nabla u\|_{L^2(D^+)}+\|F\|_{L^p(D^+)})^2.
\end{align*}
Combining the Poincar\'{e} inequality with the fact $D_{2R}(x_0)\subset D^+_{5R}(x_1)\subset D^+$, we get
\begin{align*}
\mathbb{II}\leq CR^{-2}\int_{D^+_{5R}(x_1)}|w|^2dx\leq C\int_{D^+_{5R}(x_1)}|\nabla u|^2dx.
\end{align*}
So, we have
\begin{equation}\label{equation:27}
|u(x_0)-\overline{u}|\leq C(p,N)(\|\nabla u\|_{L^2(D^+)}+\|F\|_{L^p(D^+)}).
\end{equation}

Set $d(y):=dist(y,K)$ for $y\in N$, then we have $$d(\overline{u})\leq d(u(x))+|u(x)-\overline{u}|.$$ Integrating the above inequality, we get
\begin{align*}
d(\overline{u})&\leq
\frac{1}{|D_{5R}^+(x_1)|}\int_{D_{5R}^+(x_1)}d(u(x))dx
+\frac{1}{|D_{5R}^+(x_1)|}\int_{D_{5R}^+(x_1)}|u(x)-\overline{u}|dx\\
&\leq
C(\int_{D_{5R}^+(x_1)}|\nabla (d(u(x)))|^2dx)^{1/2}
+C(\int_{D_{5R}^+(x_1)}|\nabla u|^2dx)^{1/2}\\
&\leq
C(\int_{D_{5R}^+(x_1)}|\nabla u|^2dx)^{1/2}\leq C\|\nabla u\|_{L^2(D^+)},
\end{align*}
where the second inequality follows from the Poincar\'{e} inequality since $d(u(x))=0$ on $\partial^0D_{5R}^+(x_1)$ and the third inequality follows from the fact that $Lip(d)=1$.

Then, we have
\[
dist(u(x_0),K)\leq dist(\overline{u},K)+|u(x_0)-\overline{u}|\leq C(p,N)(\|\nabla u\|_{L^2(D^+)}+\|F\|_{L^p(D^+)}),
\]
which implies \eqref{equation:25} holds.

For \eqref{equation:26}, taking $x_0=(0,\frac{1}{2})\in D^+_{\frac{1}{2}}\setminus \partial^0D^+$ in \eqref{equation:27}, then $x_1=0$, $R=\frac{1}{3}|x_0-x_1|=\frac{1}{6}$ and we get
\begin{equation}\label{equation:28}
\left|u(0,\frac{1}{2})-\frac{1}{|D^+_{\frac{5}{6}}(0)|}\int_{D^+_{\frac{5}{6}}(0)}udx\right|\leq C(p,N)(\|\nabla u\|_{L^2(D^+)}+\|F\|_{L^p(D^+)}).
\end{equation}
For any $y_0\in D^+_{\frac{1}{4}}\setminus \partial^0D^+$, set $R_{y_0}=\frac{1}{3}dist(y_0,\partial^0 D^+)$ and suppose $y_1\in\partial^0D^+$ is the nearest point
to $y_0$, $i.e.$ $|y_0-y_1|=dist(y_0,\partial^0 D^+)=3R_{y_0}$. Combing \eqref{equation:27} with \eqref{equation:28}, we obtain that
\begin{align*}
\left|u(y_0)-u(0,\frac{1}{2})\right|&\leq \left|u(y_0)-\frac{1}{|D^+_{5R_{y_0}}(y_1)|}\int_{D^+_{5R_{y_0}}(y_1)}udx\right|+ \left|u(0,\frac{1}{2})-\frac{1}{|D^+_{\frac{5}{6}}(0)|}\int_{D^+_{\frac{5}{6}}(0)}udx\right|\\
&\quad+ \left|\frac{1}{|D^+_{5R_{y_0}}(y_1)|}\int_{D^+_{5R_{y_0}}(y_1)}udx-\frac{1}{|D^+_{\frac{5}{6}}(0)|}\int_{D^+_{\frac{5}{6}}(0)}udx\right|\\
&\leq C(p,N)(\|\nabla u\|_{L^2(D^+)}+\|F\|_{L^p(D^+)})+ \left|\frac{1}{|D^+_{5R_{y_0}}(y_1)|}\int_{D^+_{5R_{y_0}}(y_1)}udx-\frac{1}{|D^+_{\frac{5}{6}}(0)|}\int_{D^+_{\frac{5}{6}}(0)}udx\right|.
\end{align*}
Noting that $D^+_{5R_{y_0}}(y_1)\subset D^+_{\frac{5}{6}}(0)$, by a variant of the classical Poincar\'{e} inequality, we have
\begin{align*}
&\left|\frac{1}{|D^+_{5R_{y_0}}(y_1)|}\int_{D^+_{5R_{y_0}}(y_1)}udx-\frac{1}{|D^+_{\frac{5}{6}}(0)|}\int_{D^+_{\frac{5}{6}}(0)}udx\right|\\
&\leq \frac{1}{|D^+_{\frac{5}{6}}(0)|} \int_{D^+_{\frac{5}{6}}(0)}\left|u-\frac{1}{|D^+_{5R_{y_0}}(y_1)|}\int_{D^+_{5R_{y_0}}(y_1)}udx\right|dx\leq C\|\nabla u\|_{L^2(D^+_{\frac{5}{6}}(0))}\leq C\|\nabla u\|_{L^2(D^+_1(0))}.
\end{align*}

Therefore,
\begin{align*}
Osc_{D^+_{\frac{1}{4}}}u:=\sup_{x,y\in D^+_{\frac{1}{4}}}|u(x)-u(y)|
&\leq \left|u(x)-u(0,\frac{1}{2})\right|+\left|u(y)-u(0,\frac{1}{2})\right|\\
&\leq C(p,N)(\|\nabla u\|_{L^2(D^+)}+\|F\|_{L^p(D^+)}).
\end{align*}

Thus, the lemma follows immediately.
\end{proof}

\

With the help of Lemma \ref{lem:coordinate}, we can extend the map to the whole disc $D$ by involuting. Firstly, we consider $1\leq k\leq n-1$. Without loss of generality, we may assume $K\cap \partial N=\emptyset$ in this paper. In fact, if $K\cap \partial N\neq\emptyset$, we  extend the target manifold $N$ smoothly across the boundary to another compact Riemannian manifold $N'$, such that $N\subset N'$ and $K\cap \partial N'=\emptyset$. Then we can consider $N'$ as a new target manifold.

Denote by $K_{\delta_0}$ the $\delta_0$-tubular neighborhood of $K$ in $N$. Taking $\delta_0>0$ small enough, then for any $y\in K_{\delta_0}$,
there exists a unique projection $y'\in K$. Set $\overline{y}=exp_{y'}\{-exp^{-1}_{y'}y\}$. So we may define an involution $\sigma$, $i.e.$ $\sigma^2=Id$ as in
\cite{Hamilton, GJ, Scheven} by
\[
\sigma(y)=\overline{y} \quad for \quad y\in K_{\delta_0}.
\]
Then it is easy to check that the linear operator $D\sigma:TN|_{K_{\delta_0}}\to TN|_{K_{\delta_0}}$ satisfies $D\sigma(V)=V$ for $V\in TK$ and $D\sigma(\xi)=-\xi$ for $\xi\in T^\perp K$.

Let $F\in L^p(D_2^+)$ for some $1<p\leq 2$ and $u\in W^{1,2}(D_2^+,N)$ be a weak solution of \eqref{weak-solution} with
free boundary $u(\partial^0D_2^+)$ on $K$. If $\|\nabla u\|_{L^2(D_2^+)}+\|F\|_{L^p(D_2^+)}\leq\epsilon_3$ where $\epsilon_3=\epsilon_3(p,N,\delta_0)>0$ is small,
by the oscillation estimate \eqref{equation:26} in Lemma \ref{lem:coordinate}, we know
\begin{equation}\label{equation:29}
u(D^+)\subset B^N_{C\epsilon_3}(u(0,\frac{1}{2}))\subset K_{\delta_0},
\end{equation}
where $B^N_{C\epsilon_3}(u(0,\frac{1}{2}))$ is the geodesic ball in $N$ with the center point $u(0,\frac{1}{2})$ and radius $ C\epsilon_3$. Then we can define an extension of $u$ to $D_1(0)$ that
\begin{align}\label{def:function}
\widehat{u}(x)=
\begin{cases}
u(x),\quad &if \quad x\in D^+;\\
\sigma(u(\rho(x))),\quad &if \quad x\in D^-,
\end{cases}
\end{align}
where $\rho(x)=(x^1,-x^2)$ for $x=(x^1,x^2)\in D_1(0)$. For $k=n$, we also use the above extension by replacing $\sigma=Id$. In the following part of this paper, we always state the argument for $1\leq k\leq n-1$, since $k=n$ is similar and easier.

At this point, one can derive the regularity at the free boundary for weak solutions of \eqref{weak-solution} by applying classical methods in \cite{GJ, Scheven} for harmonic maps or the method in \cite{CJWZ, sharpzhu} which combines the method of moving frame and some modification of Rivi\`{e}re-Struwe's method in \cite{riviere_struwe}. Now, we shall give our alternative approach which is also based on some extension of Rivi\`{e}re-Struwe's result.

In order to derive the equation of the involuted map $\widehat{u}$, we shall first define
\begin{align*}
P:B^N_{\delta_1}(u(0,\frac{1}{2}))\subset K_{\delta_0}&\to GL(\mathbb{R}^N,\mathbb{R}^N)=GL(T\mathbb{R}^N,T\mathbb{R}^N)
\end{align*}
by
\begin{align}\label{def:1}
P(y)\xi= D\sigma(y)\xi^{\top}(y)+\sum_{l={n+1}}^N\langle\xi,\nu_l(y)\rangle\nu_l(\sigma(y)),
\end{align}
where $\delta_1=\delta_1(N)$ is small such that $B^N_{4\delta_1}(u(0,\frac{1}{2}))\subset K_{\delta_0}$ and there exists a local orthonormal basis $\{\nu_l\}_{l=n+1}^N$ of the normal bundle $T^{\bot}N|_{B^N_{4\delta_1}(u(0,\frac{1}{2}))}$, $\xi^{\top}(y)$ is the projection map of $\mathbb{R}^N\to T_{y}N$. On one hand, Lemma \ref{lem:coordinate} tells us that $dist(u(0,\frac{1}{2}), K)\leq C\epsilon_3$ which implies $\sigma\left(B^N_{\delta_1}(u(0,\frac{1}{2}))\right)\subset B^N_{4\delta_1}(u(0,\frac{1}{2}))$ if we take $\epsilon_3$ small enough (e.g. $C\epsilon_3\leq \delta_1$). Thus, \eqref{def:1} is well defined. On the other hand,  noting that since \eqref{equation:29} holds, if $\epsilon_3$ is small enough (e.g. $4C\epsilon_3\leq \delta_1$), then we know that $\widehat{u}(D)\subset B^N_{4C\epsilon_3}(u(0,\frac{1}{2}))\subset B^N_{\delta_1}(u(0,\frac{1}{2}))$ and the notations $P(\widehat{u}(x))$, $O(\widehat{u}(x))$ in the sequel (see below) are well defined. It is easy to check that $P(y)$ is invertible linear operator for any $y\in B^N_{\delta_1}(u(0,\frac{1}{2}))$, since the linear operator $D\sigma(y)$ is invertible. For simplicity, we still denote by $P(y)$ the  matrix  corresponding to the linear operator $P(y)$ under the standard orthonormal basis of $\mathbb{R}^N$. Moreover, the matrix $P(y)$ and its inverse matrix $P^{-1}(y)$ are smooth for $y\in B^N_{\delta_1}(u(0,\frac{1}{2}))$. So, there exists an orthogonal matrix $O(y)$ which is also smooth, such that
\begin{align*}
O^TP^TPO=\Xi:=\begin{pmatrix}
\lambda_1(y) & 0 & 0 \\
0 & \ddots & 0 \\
0 & 0 & \lambda_N(y)
\end{pmatrix}
\end{align*}
where $P^T$ is the transposed matrix and $\lambda_i(y)$, $i=1,...,N$ is the eigenvalues of the positive symmetric matrix $P^T(y)P(y)$. It is easy to see that $\lambda_i(y)=1$ for $y\in K$, $i=1,...,N$.

Define
\begin{align*}
\rho'(x)=
\begin{cases}
x,\ x\in D^+;\\
\rho(x),\ x\in D^-,
\end{cases}
\quad and \quad
\sigma'(\widehat{u}(x))=
\begin{cases}
\widehat{u}(x),\ x\in D^+;\\
\sigma(\widehat{u}(x)),\ x\in D^-,
\end{cases}
\end{align*}
and the matrixes
\begin{align*}
Q=Q(x)=
\begin{cases}
Id_{N\times N},\ x\in D^+;\\
P(\widehat{u}(x)),\ x\in D^-,
\end{cases}
\quad and \quad
\widetilde{Q}=\widetilde{Q}(x)=
\begin{cases}
Id_{N\times N},\ x\in D^+;\\
O(\widehat{u})\sqrt{\Xi}(\widehat{u})O^T(\widehat{u}),\ x\in D^-,
\end{cases}
\end{align*}
where
\begin{align*}
\sqrt{\Xi}(y)=\begin{pmatrix}
\sqrt{\lambda_1(y)} & 0 & 0 \\
0 & \ddots & 0 \\
0 & 0 & \sqrt{\lambda_N(y)}
\end{pmatrix}.
\end{align*}
One can easily find that $\widetilde{Q}\in L^\infty\cap W^{1,2}(D,\mathbb{R}^N)$ and is invertible.

\

The involuted map satisfies the following proposition:

\

\begin{prop}\label{prop:-01}
Let $F\in L^p(D_2^+)$ for some $1<p\leq 2$ and $u(x)\in W^{1,2}(D_2^+)$ be a weak solution of \eqref{weak-solution} with free boundary $u(\partial^0D_2^+)$ on $K$. There exists a positive constant $\epsilon_3=\epsilon_3(p,N)$, such that if $\|\nabla u\|_{L^2(D_2^+)}+\|F\|_{L^p(D_2^+)}\leq\epsilon_3$ and $\widehat{u}$ is defined as above, then  $\widehat{u}\in W^{1,2}(D)$ is a weak solution of
\begin{align}\label{equation:diver-form}
 div(\widetilde{Q}\cdot\nabla\widehat{u}(x))=
\Omega\cdot\widetilde{Q}\cdot\nabla\widehat{u}(x)+\widetilde{Q}^{-1}\cdot Q^T\cdot F(\rho'(x)), \ x\in D,
\end{align}
where
\begin{align*}
\Omega(x)=
\begin{cases}
\Omega_2(x),\ x\in D^+;\\
\Omega_1(\widehat{u}(x))+\Omega_2(x)-\widetilde{Q}^{-1}\cdot \frac{1}{2}(Q^T\nabla Q-\nabla Q^TQ)\cdot \widetilde{Q}^{-1},\ x\in D^-,
\end{cases}
\end{align*}
and
\[
\Omega_1=(\Omega_1)_{AB}:=\nabla OO^T+\frac{1}{2}O\sqrt{\Xi}\nabla O^TO\sqrt{\Xi}^{-1}O^T-\frac{1}{2}O\sqrt{\Xi}^{-1}O^T \nabla O\sqrt{\Xi}O^T,
\]
\[
\Omega_2=(\Omega_2)_{AB}:=\widetilde{Q}\cdot Q^{-1}\cdot\nabla\big(\nu_l(\sigma'(\widehat{u}))\big)\cdot \nu^T_l(\widehat{u})\cdot\widetilde{Q}^{-1}-\widetilde{Q}^{-1}\cdot\nu_l(\widehat{u})
\cdot\nabla\big(\nu^T_l(\sigma'(\widehat{u}))\big)\cdot(Q^{-1})^T\cdot\widetilde{Q} ,
\]
in the distribution sense. Here, $\Omega(x)$, $\Omega_1(x)$ and $\Omega_2(x)$ are antisymmetric matrices in $L^2(D)$.

Moreover, if $u\in W^{2,p}(D^+)$, $1<p\leq 2$, then $\widehat{u}\in W^{2,p}(D)$ and satisfies \begin{equation}\label{equation:global-form}
\Delta \widehat{u}+\Upsilon_{\widehat{u}}(\nabla\widehat{u},\nabla\widehat{u})=\widehat{F}\quad in \quad D,
\end{equation}
where $\Upsilon_{\widehat{u}}(\cdot,\cdot)$ is a bounded bilinear form and $\widehat{F}\in L^p(D)$ which are defined by \eqref{def:2}, satisfying
\[
|\Upsilon_{\widehat{u}}(\nabla\widehat{u},\nabla\widehat{u})|\leq C(N)|\nabla\widehat{u}|^2\ \ and \ \ \|\widehat{F}\|_{L^p(D)}\leq C(N)\|F\|_{L^p(D^+)}.
\]

\end{prop}
\begin{proof}
\textbf{Step 1:} Firstly, it is easy to see that $\widehat{u}\in W^{1,2}(D)$. Secondly, we prove that for any arbitrary test vector field $V\in L^\infty\cap W_0^{1,2}(D,TN)$ with $V(x)\in T_{\widehat{u}(x)}N$ for $a.e.$ $x\in D$, there holds
\begin{align}\label{equation:14}
-\int_DQ\cdot\nabla\widehat{u}(x)\cdot \nabla(Q \cdot V)dx=\int_D F(\rho'(x))\cdot Q \cdot Vdx.
\end{align}

Set $\Sigma(x):=D\sigma|_{\widehat{u}(x)}$ for $x\in D$. We decompose $V$ into the symmetric and anti-symmetric part with respect to $\sigma$ as in \cite{Scheven}, $i.e.$ $V=V_e+V_a$, where
\begin{align*}
V_e(x):=\frac{1}{2}\{V(x)+\Sigma(\rho(x))V(\rho(x))\},\
V_a(x):=\frac{1}{2}\{V(x)-\Sigma(\rho(x))V(\rho(x))\}.
\end{align*}

Since $\sigma^2=Id$, we have $\Sigma(x)\Sigma(\rho(x))=Id$. Then,
\[
V_e(\rho(x))=\Sigma(x)V_e(x)\ \ and \ \  V_a(\rho(x))=-\Sigma(x)V_a(x).
\]
Noting $D\sigma:TN|_{K_{\delta_0}}\to TN|_{K_{\delta_0}}$ satisfying $D\sigma(V)=V$ for $V\in TK$ and $D\sigma(\xi)=-\xi$ for $\xi\in T^\perp K$, for any $x\in\partial^0 D^+$, we know
\begin{align*}
V_e(x)=\frac{1}{2}\{V(x)+\Sigma(x)V(x)\}=\Pi_{TK}V(x)\in TK
\end{align*}
where $\Pi_{TK}:TN\to TK$ is the orthogonal projection.

Since $u$ is a weak solution of \eqref{weak-solution} in $D^+$, we have
\begin{align}\label{equation:15}
-\int_{D^+}\nabla u(x)\nabla V_e(x)dx=\int_{D^+}F(x)\cdot V_e(x)dx.
\end{align}
Thus,
\begin{align}\label{equation:16}
-\int_{D^{-}}Q\cdot\nabla\widehat{u}(x)\cdot \nabla(Q \cdot V_e(x))dx&=
-\int_{D^{-}}D\sigma|_{\widehat{u}}\cdot\nabla\widehat{u}(x)\cdot \nabla(D\sigma|_{\widehat{u}} \cdot V_e(x))dx\notag\\
&=-\int_{D^{-}}\nabla(u(\rho(x)))\cdot \nabla(\Sigma(x) \cdot V_e(x))dx\notag\\
&=-\int_{D^{-}}\nabla(u(\rho(x)))\cdot \nabla( V_e(\rho(x)))dx\notag\\
&=-\int_{D^+}\nabla u(x)\nabla V_e(x)dx\notag\\
&=\int_{D^+}F(x)\cdot V_e(x)dx=
\int_{D^-}F(\rho'(x))\cdot Q \cdot V_e(x) dx.
\end{align}
Moreover, there holds
\begin{align}\label{equation:17}
-\int_{D^{-}}Q\cdot\nabla\widehat{u}(x)\cdot \nabla(Q \cdot V_a(x))dx&=
-\int_{D^{-}}D\sigma|_{\widehat{u}}\cdot\nabla\widehat{u}(x)\cdot \nabla(D\sigma|_{\widehat{u}} \cdot V_a(x))dx\notag\\
&=\int_{D^{-}}\nabla(u(\rho(x)))\cdot \nabla( V_a(\rho(x)))dx\notag\\
&=\int_{D^+}\nabla u(x)\nabla V_a(x)dx,
\end{align}
and
\begin{align}\label{equation:11}
\int_D F(\rho'(x))\cdot Q \cdot V_a(x)dx&=\int_{D^+} F(x) \cdot V_a(x)dx+\int_{D^-} F(\rho'(x))\cdot Q \cdot V_a(x)dx\notag\\
&=\int_{D^+} F(x) \cdot V_a(x)dx-\int_{D^-} F(\rho'(x))\cdot V_a(\rho'(x))dx\notag\\
&=\int_{D^+} F(x) \cdot V_a(x)dx-\int_{D^+} F(x) \cdot V_a(x)dx=0.
\end{align}
Then \eqref{equation:15}, \eqref{equation:16}, \eqref{equation:17} and \eqref{equation:11} imply \eqref{equation:14} immediately.

\

\textbf{Step 2:} We claim: for any $V\in L^\infty\cap W_0^{1,2}(D,\mathbb{R}^N)$, there holds
\begin{align}\label{equation:18}
&-\int_DQ\cdot\nabla\widehat{u}(x)\cdot \nabla(Q \cdot V)dx\notag\\&=-\int_{D}\langle Q\cdot\nabla\widehat{u}(x),
\nabla\big(\nu_l(\sigma'(\widehat{u}))\big)\rangle\cdot \langle \nu_l(\widehat{u}), V\rangle dx+\int_D F(\rho'(x))\cdot Q \cdot V dx.
\end{align}

In fact, on the one hand, by \eqref{equation:14}, we get
\begin{align*}
-\int_DQ\cdot\nabla\widehat{u}(x)\cdot \nabla(Q \cdot V)dx=&-\int_DQ\cdot\nabla\widehat{u}(x)\cdot \nabla(Q \cdot V^\top)dx-\int_DQ\cdot\nabla\widehat{u}(x)\cdot \nabla(Q \cdot V^\bot)dx\\
=&\int_D F(\rho'(x))\cdot Q \cdot V^\top dx-\int_DQ\cdot\nabla\widehat{u}(x)\cdot \nabla(Q \cdot V^\bot)dx.
\end{align*}
On the other hand, we have
\begin{align*}
-\int_DQ\cdot\nabla\widehat{u}(x)\cdot \nabla(Q \cdot V^\bot)dx&=
-\int_{D^+}\nabla u(x)\cdot \nabla V^\bot dx-\int_{D^-}Q\cdot\nabla\widehat{u}(x)\cdot \nabla(Q \cdot V^\bot)dx\\
&=\mathbb{I}+\mathbb{II}.
\end{align*}
Computing directly, we have
\begin{align*}
\mathbb{I}&=-\int_{D^+}\nabla u(x)\cdot \nabla(\langle V, \nu_l\rangle \nu_l)dx=-\int_{D^+}\nabla u(x)\cdot \langle V, \nu_l\rangle\nabla \nu_ldx\\&=-\int_{D^+}\langle Q\cdot\nabla\widehat{u}(x),
\nabla\big(\nu_l(\sigma'(\widehat{u}))\big)\rangle\cdot \langle  V, \nu_l(\widehat{u})\rangle dx
\end{align*}
and
\begin{align*}
\mathbb{II}&=-\int_{D^-}Q\cdot\nabla\widehat{u}(x)\cdot \nabla(Q \cdot V^\bot)dx=-\int_{D^-}Q\cdot\nabla\widehat{u}(x)\cdot \nabla\big(Q\cdot \langle V,\nu_l(\widehat{u})\rangle\nu_l(\widehat{u})\big)dx\\
&=-\int_{D^-}Q\cdot\nabla\widehat{u}(x)\cdot \nabla\big( \langle V,\nu_l(\widehat{u})\rangle\nu_l(\sigma'(\widehat{u}))\big)dx\\&=-\int_{D^-}\langle Q\cdot\nabla\widehat{u}(x),
\nabla\big(\nu_l(\sigma'(\widehat{u}))\big)\rangle\cdot \langle V, \nu_l(\widehat{u})\rangle dx.
\end{align*}

Combining these equations, we obtain
\begin{align}
&-\int_DQ\cdot\nabla\widehat{u}(x)\cdot \nabla(Q \cdot V^\bot)dx=-\int_{D}\langle Q\cdot\nabla\widehat{u}(x),
\nabla\big(\nu_l(\sigma'(\widehat{u}))\big)\rangle\cdot \langle \nu_l(\widehat{u}), V\rangle dx.
\end{align}
Thus, we have
\begin{align*}
&-\int_DQ\cdot\nabla\widehat{u}(x)\cdot \nabla(Q \cdot V)dx\\
&=-\int_{D}\langle Q\cdot\nabla\widehat{u}(x),
\nabla\big(\nu_l(\sigma'(\widehat{u}))\big)\rangle\cdot \langle \nu_l(\widehat{u}), V\rangle dx+\int_D F(\rho'(x))\cdot Q \cdot V dx,
\end{align*}
where the equality follows from that $F(\rho'(x))\in T_{u(\rho'(x))}N=T_{\sigma'(\widehat{u})}N$. This is \eqref{equation:18}.

\

\textbf{Step 3:}
In order to prove $\widehat{u}$ is a weak solution of \eqref{equation:diver-form}, take an arbitrary test vector field $V\in L^\infty\cap W_0^{1,2}(D,\mathbb{R}^N)$, since the matrix $\widetilde{Q},\widetilde{Q}^{-1}\in L^\infty\cap W^{1,2}(D,\mathbb{R}^N)$, it is sufficient to prove
\begin{align}\label{equation:22}
-\int_D\widetilde{Q}\cdot\nabla\widehat{u}(x)\cdot \nabla(\widetilde{Q} \cdot V)dx&=\int_{D}\langle \Omega\cdot\widetilde{Q}\cdot\nabla\widehat{u}(x)+\widetilde{Q}^{-1}\cdot Q^T\cdot F(\rho'(x)) ,\widetilde{Q}\cdot V\rangle dx\notag\\&=-\int_{D}\langle \widetilde{Q}\cdot\nabla\widehat{u}(x),
\Omega\cdot\widetilde{Q}\cdot V\rangle dx+\int_DF(\rho'(x))\cdot Q \cdot V dx.
\end{align}

Computing directly, we get
\begin{align*}
&-\int_{D^-}Q\cdot\nabla\widehat{u}(x)\cdot \nabla(Q \cdot V)dx\\
&=-\int_{D^-}\langle Q^TQ\cdot\nabla\widehat{u}(x),\nabla V\rangle dx
-\int_{D^-}\langle \nabla\widehat{u}(x), Q^T\nabla Q \cdot V\rangle  dx\\
&=-\int_{D^-}\langle O\sqrt{\Xi}O^T\cdot\nabla\widehat{u}(x), O\sqrt{\Xi}O^T\cdot\nabla V\rangle dx
-\int_{D^-}\langle\nabla\widehat{u}(x), \frac{1}{2}\nabla (Q^TQ) \cdot V\rangle dx\\
&\quad-\int_{D^-}\langle\nabla\widehat{u}(x), \frac{1}{2}(Q^T\nabla Q-\nabla Q^TQ) \cdot V\rangle dx\\
&=-\int_{D^-}\langle O\sqrt{\Xi}O^T\cdot\nabla\widehat{u}(x), \nabla(O\sqrt{\Xi}O^T\cdot V)\rangle dx-\int_{D^-}\langle\nabla\widehat{u}(x), \frac{1}{2}(Q^T\nabla Q-\nabla Q^TQ) \cdot V\rangle dx\\
&\quad+\int_{D^-}\langle\widetilde{Q}\cdot\nabla\widehat{u}(x), \left(\nabla(O\sqrt{\Xi}O^T)-\widetilde{Q}^{-1}\cdot\frac{1}{2}\nabla (Q^TQ)\right) \cdot V\rangle dx,
\end{align*}
and
\begin{align*}
&\left(\nabla(O\sqrt{\Xi}O^T)-\widetilde{Q}^{-1}\cdot\frac{1}{2}\nabla (Q^TQ)\right) \cdot \widetilde{Q}^{-1}\\&=\nabla OO^T+\frac{1}{2}O\sqrt{\Xi}\nabla O^TO\sqrt{\Xi}^{-1}O^T-\frac{1}{2}O\sqrt{\Xi}^{-1}O^T \nabla O\sqrt{\Xi}O^T:=\Omega_1,
\end{align*}
where $\Omega_1$ is an antisymmetric matrix since $O^TO=OO^T=Id$.

Noting that $Q(x)=\widetilde{Q}(x)$, $x\in D^+$, thus, we have
\begin{align*}
&-\int_{D}Q\cdot\nabla\widehat{u}(x)\cdot \nabla(Q \cdot V)dx\\
&=-\int_D\widetilde{Q}\cdot\nabla\widehat{u}(x)\cdot \nabla(\widetilde{Q} \cdot V)dx-\int_{D^-}\langle\nabla\widehat{u}(x), \frac{1}{2}(Q^T\nabla Q-\nabla Q^TQ) \cdot V\rangle dx\\
&\quad+\int_{D^-}\langle\widetilde{Q}\cdot\nabla\widehat{u}(x), \Omega_1\cdot\widetilde{Q} \cdot V\rangle dx.
\end{align*}

By \eqref{equation:18}, we get
\begin{align}\label{equation:23}
&-\int_D\widetilde{Q}\cdot\nabla\widehat{u}(x)\cdot \nabla(\widetilde{Q} \cdot V)dx\notag\\&=\int_{D^-}\langle\nabla\widehat{u}(x), \frac{1}{2}(Q^T\nabla Q-\nabla Q^TQ) \cdot V\rangle dx
-\int_{D^-}\langle\widetilde{Q}\cdot\nabla\widehat{u}(x), \Omega_1\cdot\widetilde{Q} \cdot V\rangle dx\notag\\&\quad-\int_{D}\langle Q^TQ\cdot\nabla\widehat{u}(x),
Q^{-1}\nabla\big(\nu_l(\sigma'(\widehat{u}))\big)\rangle\cdot \langle \nu_l(\widehat{u}), V\rangle dx+\int_D F(\rho'(x))\cdot Q \cdot V dx.
\end{align}
Noting that $\widetilde{Q}^T=\widetilde{Q}$ and $$\langle \widetilde{Q}\cdot\nabla\widehat{u}(x),\widetilde{Q}^{-1}\cdot\nu_l(\widehat{u})\rangle=0,$$ we have
\begin{align*}
&-\int_{D}\langle Q^TQ\cdot\nabla\widehat{u}(x),
Q^{-1}\nabla\big(\nu_l(\sigma'(\widehat{u}))\big)\rangle\cdot \langle \nu_l(\widehat{u}), V\rangle dx\\
&=-\int_{D}\langle \widetilde{Q}\cdot\nabla\widehat{u}(x),
\widetilde{Q}\cdot Q^{-1}\nabla\big(\nu_l(\sigma'(\widehat{u}))\big)\rangle\cdot \langle \widetilde{Q}^{-1}\cdot\nu_l(\widehat{u}), \widetilde{Q}\cdot V\rangle dx\\
&=-\int_{D}\langle \widetilde{Q}\cdot\nabla\widehat{u}(x),
\Omega_2\cdot\widetilde{Q}\cdot V\rangle dx.
\end{align*}

Thus, \eqref{equation:23} implies
\begin{align*}
&-\int_D\widetilde{Q}\cdot\nabla\widehat{u}(x)\cdot \nabla(\widetilde{Q} \cdot V)dx\notag\\&=\int_{D^-}\langle\widetilde{Q}\cdot\nabla\widehat{u}(x), \frac{1}{2}\widetilde{Q}^{-1}\cdot(Q^T\nabla Q-\nabla Q^TQ) \cdot\widetilde{Q}^{-1}\cdot \widetilde{Q}\cdot V\rangle dx
-\int_{D^-}\langle\widetilde{Q}\cdot\nabla\widehat{u}(x), \Omega_1\cdot\widetilde{Q} \cdot V\rangle dx\notag\\&\quad-\int_{D}\langle \widetilde{Q}\cdot\nabla\widehat{u}(x),
\Omega_2\cdot\widetilde{Q}\cdot V\rangle dx+\int_D F(\rho'(x))\cdot Q \cdot V dx.
\end{align*}

This is \eqref{equation:22}. We proved the first result of the lemma.

\

\textbf{Step 4:} If $u\in W^{2,p}(D^+)$, according to the property of $D\sigma$, it is easy to see $\widehat{u}\in W^{2,p}(D)$ since $u$ satisfies free boundary condition. Computing directly, we have
\begin{align*}
\nabla_{e_\alpha}\widehat{u}(x)&=D\sigma|_{u(\rho(x))}\circ Du|_{\rho(x)}\circ D\rho|_x(e_\alpha)\\
&=D\sigma|_{u(\rho(x))}\circ D\Pi_N|_{u(\rho(x))} \circ Du|_{\rho(x)}\circ D\rho|_x(e_\alpha), \quad x\in D^-,
\end{align*}
where $\Pi_N:N_{\delta_0'}\to N$ is the nearest projection map for some $\delta_0'-$neighborhood of $N$ in $\mathbb{R}^N$.

By direct computing, we obtain
\begin{align*}
\Delta\widehat{u}(x)
&= D^2(\sigma\circ\Pi_N)|_{\sigma(\widehat{u})}(\nabla(u\circ \rho),
\nabla(u\circ \rho))+D\sigma(\sigma(\widehat{u}))\cdot F(\rho(x))\\
&= D^2(\sigma\circ\Pi_N)|_{\sigma(\widehat{u})}(P(\widehat{u})\cdot\nabla\widehat{u}(x),
P(\widehat{u})\cdot\nabla\widehat{u}(x))+P(\sigma(\widehat{u}))\cdot F(\rho(x)).
\end{align*}
Combining this with the fact that $\widehat{u}$ satisfies equation \eqref{weak-solution} in $D^+$,  the equation \eqref{equation:global-form} follows immediately by taking
\begin{align}\label{def:2}
\Upsilon_{\widehat{u}}(\cdot,\cdot)=
\begin{cases}
A(\widehat{u})(\cdot,\cdot)\ in\ D^+,\\
D^2(\sigma\circ\Pi_N)|_{\sigma(\widehat{u})}(P(\widehat{u})\cdot,
P(\widehat{u})\cdot)\ in\ D^-;
\end{cases}
and \quad
\widehat{F}=
\begin{cases}
F(x)\ in\ D^+,\\
P(\sigma(\widehat{u}))\cdot F(\rho(x))\ in\ D^-.
\end{cases}
\end{align}
\end{proof}

Now, applying Theorem \ref{thm:regularity-A}, we derive the following

\begin{thm}\label{thm:3}
Let $F\in L^p(D_2^+)$ for some $p>1$ and $u\in W^{1,2}(D_2^+,N)$ be a weak solution of \eqref{weak-solution}
with free boundary $u(\partial^0D_2^+)$ on $K$. Suppose $\|\nabla u\|_{L^2(D_2^+)}+\|\tau(u)\|_{L^p(D_2^+)}\leq\epsilon_3$,
then $u(x)\in W^{2,p}(D_{\frac{1}{2}}^+)$.
\end{thm}
\begin{proof}
By Proposition \ref{prop:-01}, the extended $\widehat{u} \in W^{1,2}(D, \R^N)$ is a weak solution to a system \eqref{regularity-A} with $A$ satisfying \eqref{Abound}
and with $\Omega$ satisfying $|\Omega|\leq C |\nabla \widehat{u}|$. Then we can apply Theorem \ref{thm:regularity-A} (taking $\epsilon_3$ smaller if necessary) for $1<p<2$ and bootstrap for $p\geq2$ to prove the theorem.

\end{proof}

Moreover, we have

\begin{thm}\label{thm:regularity}
Let $M$ be a compact Riemann surface with smooth boundary $\partial M$, $N$ a compact Riemannian manifold, and $K\subset N$ a smooth submanifold.
Let $F\in L^p(M)$ for some $p>1$, and $u\in H^1(M,N)$ be a weak solution of \eqref{weak-solution} with free boundary $u(\partial M)$ on $K$,
then $u\in W^{2,p}(M)$.
\end{thm}

To end this section, we derive the removability of a local  singularity at the free boundary (see Theorem \ref{thm:remov-sing-interior} for the interior case).

\begin{thm}\label{thm:remov-sing-boundary}
Let $u\in W^{2,p}_{loc}(D^+\setminus\{0\},N)$, $p>1$ be a map with finite energy that satisfies
\begin{align}
&\tau(u)=g\in L^p(D^+,TN),\quad a.e.\ x\in D^+,\label{equation:12}\\
&u(x)\in K,\quad du(x)(\overrightarrow{n})\perp T_{u(x)}K, \quad a.e.\ x\in \partial^0D^+,\label{equation:13}
\end{align}
then $u$ can be extended to a map belonging to $W^{2,p}(D^+,N)$.
\end{thm}
\begin{proof}
Applying a similar argument as in Lemma A.2 in \cite{Jost1991}, it is easy to see that $u$ is a weak solution of \eqref{weak-solution} with $F=g$ and
with free boundary $u(\partial^0D^+)$ on $K$. By Theorem \ref{thm:3}, we know $u\in W^{2,p}(D^+_r)$ for some small $r>0$. Thus, $u\in W^{2,p}(D^+)$.
\end{proof}

\

\section{Some basic analytic properties for the free boundary case}\label{sec:basic-lemma}

\

In this section, we will prove some basic lemmas for the free boundary case, such as small energy regularity (near the boundary), gap theorem, Pohozaev identity.

Firstly, we prove a small energy regularity lemma near the boundary.
\begin{lem}\label{lem:small-energy-regularity}
Let $u\in W^{2,p}(D_2^+,N)$, $1<p\leq 2$ be a map with tension field $\tau(u)\in L^p(D_2^+)$ and with free boundary $u(\partial^0D_2^+)$ on $K$.
There exists $\epsilon_4=\epsilon_4(p,N)>0$, such that if $\|\nabla u\|_{L^2(D_2^+)}+\|\tau(u)\|_{L^p(D_2^+)}\leq\epsilon_4$, then
\begin{equation}
\|u-\frac{1}{|D^+|}\int_{D^+}udx\|_{W^{2,p}(D_{1/2}^+)}\leq C(p,N)(\|\nabla u\|_{L^p(D^+)}+\|\tau(u)\|_{L^p(D^+)}).
\end{equation}
Moreover, by the Sobolev embedding $W^{2,p}(\mathbb{R}^2)\subset C^0(\mathbb{R}^2)$, we have
\begin{equation}
\|u\|_{Osc(D_{1/2}^+)}=\sup_{x,y\in D_{1/2}^+}|u(x)-u(y)|\leq C(p,N)(\|\nabla u\|_{L^p(D^+)}+\|\tau(u)\|_{L^p(D^+)}).
\end{equation}
\end{lem}
\begin{proof}
By Proposition \ref{prop:-01}, we can extend $u$ to $\widehat{u}\in W^{2,p}(D)$ which is defined in $D$ and satisfies
\begin{equation}
\triangle\widehat{u}+\Upsilon_{\widehat{u}}(\nabla\widehat{u},\nabla\widehat{u})=\widehat{F}\quad in \quad D.
\end{equation}
where $F=\tau(u)$ in $D^+$ and $\Upsilon_{\widehat{u}}(\cdot,\cdot)$, $\widehat{F}(x)$ are defined by \eqref{def:2}.

Firstly, we let $1<p<2$. Take a cut-off function $\eta\in C^\infty_0(D)$, such that $0\leq\eta\leq 1$, $\eta|_{D_{3/4}}\equiv 1$ and $|\nabla\eta|\leq C$. Then, we have
\begin{align*}
\Delta(\eta \widehat{u})=\eta\Delta\widehat{u}+2\nabla\eta\nabla\widehat{u}+\widehat{u}\Delta\eta
\leq
C(N)|\nabla\widehat{u}||\nabla(\eta\widehat{u})|
+C(N)(|\nabla\widehat{u}|+|\widehat{u}|+|\widehat{F}|).
\end{align*}
Without loss of generality, we may assume $\frac{1}{|D^+|}\int_{D^+}\widehat{u}dx=\frac{1}{|D^+|}\int_{D^+}udx=0$. By the standard elliptic estimates, Sobolev's embedding, Poincar\'{e}'s inequality and Proposition \ref{prop:-01}, we have
\begin{align*}
\|\eta \widehat{u}\|_{W^{2,p}(D)}&\leq
C(p,N)\||\nabla\widehat{u}||\nabla(\eta\widehat{u})|\|_{L^p(D)}
+C(p,N)(\|\nabla\widehat{u}\|_{L^p(D)}+\|\widehat{u}\|_{L^p(D)}
+\|\widehat{F}\|_{L^p(D)})\\
&\leq
C(p,N)\|\nabla\widehat{u}\|_{L^2(D)}\|\nabla(\eta\widehat{u})\|_{L^{\frac{2p}{2-p}}(D)}
+C(p,N)(\|\nabla\widehat{u}\|_{L^p(D)}
+\|\tau(u)\|_{L^p(D^+)})\\
&\leq
C(p,N)\epsilon_4\|\eta \widehat{u}\|_{W^{2,p}(D)}
+C(p,N)(\|\nabla u\|_{L^p(D^+)}
+\|\tau(u)\|_{L^p(D^+)}),
\end{align*}
where we also used the fact that $\|\nabla\widehat{u}\|_{L^p(D)}\leq C(N)\|\nabla u\|_{L^p(D^+)}$, $1< p\leq 2$.

Taking $\epsilon_4$ sufficiently small, we have
\begin{align*}
\|u\|_{W^{2,p}(D^+_{1/2})}\leq\|\eta \widehat{u}\|_{W^{2,p}(D)}\leq
C(p,N)(\|\nabla u\|_{L^p(D^+)}
+\|\tau(u)\|_{L^p(D^+)}).
\end{align*}

So, we have proved the lemma in the case $1<p<2$.
Next, if $p=2$, one can first derive the above estimate with $p=\frac{4}{3}$. Such an estimate gives a $L^4(D^+_{3/4})-$ bound for $\nabla u$. Then one can apply the $W^{2,2}-$boundary estimate to the equation and get the conclusion of lemma with $p=2$.
\end{proof}

The gap theorem still holds for harmonic maps with free boundary.
\begin{lem}\label{lem:gap-theory}
There exists a constant $\epsilon_5=\epsilon_5(M,N)>0$ such that if $u$ is a smooth harmonic map from
$M$ to $N$ with free boundary on $K$ and satisfying $$\int_M|\nabla u|^2dvol\leq\epsilon_5,$$ then $u$ is a constant map.
\end{lem}
\begin{proof}
By Lemma \ref{lem:coordinate}, Lemma \ref{lem:small-energy-regularity} and Lemma \ref{lem:small-energy-regularity-interior}, take any $x_0\in M$,
then we may assume the image of $u$ is contained in a Fermi-coordinate chart $(B_{R_0}^N(u(x_0)),y^i)$ of $N$. Thus, we can rewrite the equation
in the new coordinate as follows:
\begin{eqnarray*}
\begin{cases}
-\Delta_M u+\Gamma(u)(\nabla u,\nabla u)=0, \ in\ M;\\
\frac{\partial u^i(x)}{\partial \overrightarrow{n}}=0,\quad 1\leq i\leq k,\quad u^j(x)=0,\quad k+1\leq j\leq n,\quad x\in\partial M.
\end{cases}
\end{eqnarray*}
where $\Gamma(u)(\nabla u,\nabla u)=g^{\alpha\beta}\Gamma^i_{jk}(u)\frac{\partial u^j}{\partial x^\alpha}\frac{\partial u^k}{\partial x^\beta}\frac{\partial}{\partial y^i}$ and $\Gamma^i_{jk}$ are the Christoffel symbol of $N$ in local coordinates $\{y^i\}_{i=1}^n$.

Without loss of generality, we may assume $\int_{M}u^i=0$, $1\leq i\leq k$. By standard elliptic estimates with Dirichlet boundary condition and Neumann boundary condition (see Lemma \ref{lem:lp-estimate-neumann}), we have
\begin{align*}
\|\nabla u\|_{W^{1,4/3}(M)}&\leq C(M)\|\Delta_M u\|_{L^{4/3}(M)}\\
&\leq C(M,N)\|\nabla u\|_{L^2(M)}\|\nabla u\|_{L^{4}(M)}\\
&\leq C(M,N)\sqrt{\epsilon_5}\|\nabla u\|_{L^{4}(M)}\leq C(M,N)\sqrt{\epsilon_5}\|\nabla u\|_{W^{1,4/3}(M)}.
\end{align*}
If $\epsilon_5$ is small, then $u$ is a constant map.
\end{proof}

Next, we compute the Pohozaev identity which is similar to \cite{Lin-Wang}.
\begin{lem}\label{Pohozaev}
For $x_0\in \partial^0D^+$, let $u(x)\in W^{2,2}(D^+(x_0),N)$ be a map with tension field $\tau(u)\in L^2(D^+(x_0))$
and with free boundary $u(\partial^0D^+)$ on $K$. Then, for any $0<t<1$, there holds
\begin{align}\label{inequality:13}
\int_{\partial^+D_t^+(x_0)}r(|\frac{\partial u}{\partial r}|^2-\frac{1}{2}|\nabla u|^2)=\int_{D_t^+(x_0)}r\frac{\partial u}{\partial r}\tau dx
\end{align}
where $(r,\theta)\in (0,1)\times (0,\pi)$ are the polar coordinates at $x_0$.
\end{lem}
\begin{proof}
Since $u(x)$ satisfies the equation
\[
\tau=\Delta u+A(u)(\nabla u,\nabla u)  \quad a.e.\ x\in D^+(x_0)
\]
with the free boundary $u(\partial^0D^+)$ on $K$, multiplying $(x-x_0)\nabla u$ to both sides of the above equation and integrating by parts, for any $0<t<1$, we get
\bee
&&\int_{D_t^+(x_0)}\tau\cdot((x-x_0)\nabla u)dx\notag\\
&=&\int_{D_t^+(x_0)}\Delta u\cdot((x-x_0)\nabla u)dx\notag\\
&=&\int_{\partial (D_t^+(x_0))}\frac{\partial u}{\partial n}\cdot((x-x_0)\nabla u)-\int_{D_t^+(x_0)}\nabla_{e_\alpha}u\cdot\nabla_{e_\alpha}((x-x_0)\nabla u)dx\notag\\
&=&\int_{\partial^+ (D_t^+(x_0))}\frac{\partial u}{\partial n}\cdot((x-x_0)\nabla u)-\int_{D_t^+(x_0)}|\nabla u|^2dx-\frac{1}{2}\int_{D_t^+(x_0)}(x-x_0)\cdot\nabla|\nabla u|^2dx\notag\\
&=&\int_{\partial^+ (D_t^+(x_0))}\frac{\partial u}{\partial n}\cdot((x-x_0)\nabla u)-\frac{1}{2}\int_{\partial (D_t^+(x_0))}\langle x-x_0,\overrightarrow{n}\rangle|\nabla u|^2\notag\\
&=&\int_{\partial^+ (D_t^+(x_0))}\frac{\partial u}{\partial n}\cdot((x-x_0)\nabla u)
-\frac{1}{2}\int_{\partial^+ (D_t^+(x_0))}\langle x-x_0,\overrightarrow{n}\rangle|\nabla u|^2\notag\\
&=&\int_{\partial^+ (D_t^+(x_0))}r\left(|\frac{\partial u}{\partial r}|^2-\frac{1}{2}|\nabla u|^2\right),\notag
\eee
where the last second equality follows from the fact that $\langle x-x_0,\overrightarrow{n}\rangle=0$ on $\partial^0D_t^+(x_0)$. Then the conclusion of lemma follows immediately.
\end{proof}

\begin{cor}\label{cor:01}
Under the assumptions of Lemma \ref{Pohozaev}, we have
\begin{align*}
\int_{D_{2t}^+(x_0)\setminus D_{t}^+(x_0)}(|\frac{\partial u}{\partial r}|^2-\frac{1}{2}|\nabla u|^2)dx\leq t\|\nabla u\|_{L^2(D^+(x_0))}\|\tau \|_{L^2(D^+(x_0))}.
\end{align*}
\end{cor}
\begin{proof}
From Lemma \ref{Pohozaev}, we have
\begin{align*}
\int_{\partial^+D_t^+(x_0)}(|\frac{\partial u}{\partial r}|^2-\frac{1}{2}|\nabla u|^2)=\int_{D_t^+(x_0)}\frac{r}{t}\frac{\partial u}{\partial r}\tau dx\leq \|\nabla u\|_{L^2(D^+(x_0))}\|\tau \|_{L^2(D^+(x_0))}.
\end{align*}
Integrating from $t$ to $2t$, we will get the conclusion of the corollary.
\end{proof}

\

\section{Energy identity and no neck property} \label{sec:energy-identity}

\

In this section, we shall prove our main Theorem \ref{main:thm-1}.

We first consider the following simpler case of a single boundary blow-up point.

\begin{thm}\label{thm:1}
Let $u_n \in W^{2,2}(D_1^+(0),N)$ be a sequence of maps with tension fields $\tau(u_n)$
and with free boundaries $u_n(\partial^0D^+)$ on $K$ and satisfying
\begin{itemize}
\item[(a)]  $\ \|u_n\|_{W^{1,2}(D^+)}+\|\tau(u_n)\|_{L^{2}(D^+)}\leq \Lambda,$ \\
\item[(b)] $\ u_n\to u \mbox{ strongly in }W_{loc}^{1,2}(D^+\setminus\{0\},\mathbb{R}^N)\ as\ n\to\infty,$ \\
\item[(c)]  $\ u_n(x)\in K,\quad du_n(x)(\overrightarrow{n})\perp T_{u_n(x)}K, \quad x\in \partial^0 D^+$.
\end{itemize}
Then there exist a subsequence of $u_n$ (still denoted by $u_n$) and a nonnegative integer $L$ such that, for any $i=1,...,L$, there exist a point $x^i_n$,
positive numbers $\lambda^i_n$ and a nonconstant harmonic sphere $w^i$ or a nonnegative constant $a^i$ and
a nonconstant harmonic disk $w^i$ (which we view as a map from $\mathbb{R}_{a^i}^2\cup\{\infty\}\to N$) with free boundary
$w^i(\partial \mathbb{R}_{a^i}^2)$ on $K$ such that:
\begin{itemize}
\item[(1)]  $\ x^i_n\to 0,\ \lambda^i_n\to 0$, as $n\to\infty$;\\
\item[(2)]  $\ \frac{dist(x^i_n,\partial^0D^+)}{\lambda^i_n}\to a^i$ or $\frac{dist(x^i_n,\partial^0D^+)}{\lambda^i_n}\to\infty\ (i.e. \ a^i=\infty)$, as $n\to\infty$;\\
\item[(3)]  $\ \lim_{n\to\infty}\big(\frac{\lambda^i_n}{\lambda^j_n}+\frac{\lambda^j_n}{\lambda^i_n} +\frac{|x^i_n-x^j_n|}{\lambda^i_n+\lambda^j_n}\big)=\infty$ for any $i\neq j$;\\
\item[(4)] $\ w^i$ is the weak limit of $u_n(x^i_n+\lambda^i_nx)$ in $W^{1,2}_{loc}(\mathbb{R}^2)$, if $\frac{dist(x^i_n,\partial^0D^+)}{\lambda^i_n}\to\infty$
or $w^i$ is the weak limit of $u_n(x^i_n+\lambda^i_nx)$ in $W^{1,2}_{loc}(\mathbb{R}_{a^i}^{2+})$, if $\frac{dist(x^i_n,\partial^0D^+)}{\lambda^i_n}\to a^i$;\\

\item[(5)]  \textbf{Energy identity:} we have
\begin{equation}
\lim_{n\to\infty}E(u_n,D^+)=E(u,D^+)+\sum_{i=1}^LE(w^i).
\end{equation} \\

\item[(6)]  \textbf{No neck property:} The image
\begin{align}
u(D^+)\cup\bigcup_{i=1}^Lw^i(\mathbb{R}^2_{a^i})
\end{align}
is a connected set, where $w^i(\mathbb{R}^2_{a^i})=w^i(\mathbb{R}^2)$,  if $\frac{dist(x^i_n,\partial^0D^+)}{\lambda^i_n}\to\infty$.
\end{itemize}
\end{thm}

\begin{proof}[\textbf{Proof of Theorem \ref{thm:1}}]
Assume $0$ is the only blow-up point of the sequence $\{u_n\}$ in $D^+$, $i.e.$
\begin{equation}\label{def:small-energy}
\liminf_{n\to\infty}E(u_n;D^+_r)\geq \frac{\overline{\epsilon}^2}{8}\mbox{ for all }r>0
\end{equation}
where $\overline{\epsilon}=\min\{\epsilon_1,\epsilon_3,\epsilon_4\}$. By the standard argument of blow-up analysis we can assume that, for any $n$, there exist sequences $x_n\to 0$ and $r_n\to 0$ such that
\begin{equation}\label{equation:02}
E(u_n;D^+_{r_n}(x_n))=\sup_{\substack{x\in D^+,r\leq r_n\\D^+_r(x)\subset D^+}}E(u_n;D^+_r(x))=\frac{\overline{\epsilon}^2}{32}.
\end{equation}

Denoting $d_n=dist(x_n,\partial^0D^+)$, we have the following two cases:

\

\noindent\textbf{Case 1:} $\limsup_{n\to\infty}\frac{d_n}{r_n}<\infty$.

Set
\[
v_n(x):=u_n(x_n+r_nx)
\]
and
\[
B_n:=\{x\in\mathbb{R}^2|x_n+r_nx\in D^+\}.
\]
After taking a subsequence, we may assume $\lim_{n\to\infty}\frac{d_n}{r_n}=a\geq0$. Then
\[
B_n\to \mathbb{R}^2_a:=\{(x^1,x^2)|x^2\geq -a\}.
\]
It is easy to see that $v_n(x)$ is defined in $B_n$ and satisfies
\begin{align}
&\tau(v_n(x))=\Delta v_n(x)+A(v_n(x))(\nabla v_n(x),\nabla v_n(x))\quad x\in B_n;\\
&v_n(x)\in K,\quad dv_n(x)(\overrightarrow{n})\perp T_{v_n(x)}K, \mbox{ if } x_n+r_nx\in\partial^0D^+,
\end{align}
where $\tau(v_n(x))=r_n^2\tau(u_n(x_n+r_nx))$.

Noting that for any $x\in\partial^0B_n:=\{x\in\mathbb{R}^2| \ x_n+r_nx\in\partial^0D^+\}$ on the boundary,
\begin{align*}
v_n(x)\in K,\quad dv_n(x)(\overrightarrow{n})\perp T_{v_n(x)}K,
\end{align*}
since $\|\tau(v_n)\|_{L^2( B_n)}\leq r_n\|\tau(u_n)\|_{L^2(D^+)}\leq \frac{\overline{\epsilon}^2}{4}$ when $n$ is big enough, by \eqref{equation:02} and Lemma \ref{lem:small-energy-regularity}, we have
\begin{equation}\label{inequality:02}
\|v_n\|_{W^{2,2}(D_{4R}(0)\cap B_n)}\leq C(R,N)
\end{equation}
for any $D_{4R}(0)\subset \mathbb{R}^2$. Then there exist a subsequence of $v_n$ (also denoted by $v_n$) and a nontrivial harmonic map $\widetilde{v}^1\in W^{2,2}(\mathbb{R}_a^2)$ with free boundary $\widetilde{v}^1(\partial\mathbb{R}_a^2)$ on $K$ such that for any $R>0$, there hold
\begin{equation}\label{equation:10}
\lim_{n\to\infty}\|v_n(x)-\widetilde{v}^1(x)\|_{W^{1,2}(D_R(0)\cap B_n\cap \R^2_a)}=0
\end{equation}
and
\begin{equation}\label{equation:09}
\lim_{n\to\infty}\|v_n(x)\|_{W^{1,2}(D_R(0)\cap B_n)}=\|\widetilde{v}^1(x)\|_{W^{1,2}(D_R(0)\cap \R^2_a)}.
\end{equation}
In fact, by \eqref{inequality:02}, we have
\begin{equation}
\|v_n\big(x-(0,\frac{d_n}{r_n})\big)\|_{W^{2,2}(D^+_{3R}(0))}\leq C(R,N)
\end{equation}
when $n$ is big enough. Then there exist a subsequence of $v_n$ (also denoted by $v_n$) and a harmonic map $\widetilde{v}\in W^{2,2}(D^+_{3R}(0))$ such that $$\lim_{n\to \infty}\|v_n\big(x-(0,\frac{d_n}{r_n})\big)-\widetilde{v}(x)\|_{W^{1,2}(D^+_{3R}(0))}=0$$ and
$v_n\big(x-(0,\frac{d_n}{r_n})\big)\to\widetilde{v}(x)$, $\frac{dv_n\big(x-(0,\frac{d_n}{r_n})\big)}{d\overrightarrow{n}}\to \frac{d\widetilde{v}}{d\overrightarrow{n}}(x)$, $a.e.\ x\in \partial^0D^+_{3R}(0)$ as $n\to\infty$. Set
$$\widetilde{v}^1(x):=\widetilde{v}(x+(0,a)),$$
then $\widetilde{v}^1\in W^{2,2}(\mathbb{R}_a^2\cap D_{2R}(0))$ is a harmonic map with free boundary $\widetilde{v}^1(\partial\mathbb{R}_a^2\cap D_{2R}(0))$ on $K$ such that $$\lim_{n\to \infty}\|v_n(x)-\widetilde{v}^1(x)\|_{W^{1,2}(D_{2R}(0)\cap B_n\cap \R^2_a)}=0.$$ Lastly, \eqref{equation:09} follows from \eqref{inequality:02}, \eqref{equation:10}, Sobolev embedding, Young's inequality and the fact that the measure of $D_{2R}(0)\cap B_n\setminus \R^2_a$ goes to zero as $n\to\infty$.

In addition,  $E(\widetilde{v}^1;D_1(0)\cap\mathbb{R}_a^2)=\frac{\overline{\epsilon}^2}{32}$. By the conformal invariance of harmonic maps and the removable singularity Theorem \ref{thm:remov-sing-boundary}, $\widetilde{v}^1(x)$ can be extended to a nontrivial harmonic disk.

\

\noindent\textbf{Case 2:} $\limsup_{n\to\infty}\frac{d_n}{r_n}=\infty$.

In this case, we can see that $v_n(x)$ is defined in $B_n$ which tends to $\mathbb{R}^2$ as $n\to\infty$.
Moreover, for any $x\in\mathbb{R}^2$, when $n$ is sufficiently large, by \eqref{equation:02}, we have
\begin{equation}
E(v_n;D_1(x))\leq \frac{\overline{\epsilon}^2}{32}.
\end{equation}
According to Lemma \ref{lem:small-energy-regularity-interior}, there exist a subsequence of $v_n$ (we still denote it by $v_n$) and a harmonic map $v^1(x)\in W^{1,2}(\mathbb{R}^2,N)$ such that
\[
\lim_{n\to\infty}v_n(x)=v^1(x) \mbox{ in }W^{1,2}_{loc}(\mathbb{R}^2).
\]
Besides, we know $E(v^1;D_1(0))=\frac{\overline{\epsilon}^2}{32}$. By the standard theory of harmonic maps, $v^1(x)$ can be extended to a nontrivial harmonic sphere. We call the above harmonic sphere $v^1(x)$ or harmonic disk $\widetilde{v}^1(x)$ the first bubble.

\

We will split the proof of Theorem \ref{thm:1} into two parts, energy identity and no neck   result. Now, we begin to prove the energy identity.

\

\noindent\textbf{Energy identity }:
By the standard induction argument in \cite{DingWeiyueandTiangang}, we only need to prove the theorem in the case where there is only one bubble.

\

By Lemma \ref{lem:small-energy-regularity-interior} and Lemma \ref{lem:small-energy-regularity}, there exist a subsequence of $u_n$ (still denoted by $u_n$) and a weak limit $u\in W^{2,2}(D^+)$ such that
\[
\lim_{\delta\to 0}\lim_{n\to \infty}E(u_n;D^+\setminus D^+_\delta(x_n))=E(u;D^+).
\]
So, in both cases, the energy identity is equivalent to
\begin{equation}\label{equation:energy-identity}
\lim_{R\to\infty}\lim_{\delta\to 0}\lim_{n\to \infty}E(u_n;D^+_\delta(x_n)\setminus D^+_{r_nR}(x_n))=0.
\end{equation}
To prove the no neck   property, $i.e.$ that  the sets $u(D^+)$ and $v(\mathbb{R}^2\cup\infty)$ or $v(\mathbb{R}_a^2\cup\infty)$ are connected, it is enough to show that
\begin{equation}\label{equation:no-neck}
\lim_{R\to\infty}\lim_{\delta\to 0}\lim_{n\to \infty}\|u_n\|_{Osc\big(D^+_\delta(x_n)\setminus D^+_{r_nR}(x_n)\big)}=0.
\end{equation}

\

\noindent\textbf{Step 1:} We prove the energy identity for \textbf{Case 1}, $i.e.$, $\lim_{n\to\infty}\frac{d_n}{r_n}=a<\infty$.

\

Under the ``one bubble" assumption, we first make the following:

\

\noindent\textbf{Claim:} for any $\epsilon>0$, there exist $\delta>0$ and $R>0$ such that
\begin{equation}\label{equation:assumption-small}
\int_{D^+_{8t}(x_n)\setminus D^+_{t}(x_n)}|\nabla u_n|^2dx\leq \epsilon^2 \mbox{ for any } t\in(\frac{1}{2}r_nR,2\delta)
\end{equation}
when $n$ is large enough.

In fact, if \eqref{equation:assumption-small} is not true, then we can find $t_n\to 0$, such that $\lim_{n\to \infty}\frac{t_n}{r_n}=\infty$ and
\begin{equation}\label{equation:19}
\int_{D^+_{8t_n}(x_n)\setminus D^+_{t_n}(x_n)}|\nabla u_n|^2dx\geq \epsilon_6>0.
\end{equation}
 Then we have \[
\lim_{n\to\infty}\frac{d_n}{t_n}=0.
\]

We set $$w_n(x):=u_n(x_n+t_nx)$$ and $$B'_n:=\{x\in\mathbb{R}^2|x_n+t_nx\in D^+\}.$$ Then $w_n(x)$ lives in $B'_n$ which tends to $\mathbb{R}_+^2$ as $n\to\infty$. It is easy to see that $0$ is an energy concentration point for $w_n$. We have to consider the following two cases:

\

\noindent$\textbf{(a)}$ $w_n$ has no other energy concentration points except $0$.

\

By Lemma \ref{lem:small-energy-regularity-interior}, Lemma \ref{lem:small-energy-regularity} and the process of constructing the first bubble, passing to a subsequence, we may assume that  $w_n$ converges to a harmonic map $w(x):\mathbb{R}^2_+\to N$ with free boundary $w(\partial \mathbb{R}^2_+)$ on $K$ satisfying, for any $R>0$,
\[
\sup_{\lambda>0}\lim_{n\to\infty}\|w_n(x)-w(x)\|_{W^{1,2}\big((D_R(0)\cap B'_n)\setminus D_\lambda(0)\big)}=0.
\]
Noting that \eqref{equation:19} implies
\begin{equation}
\int_{(D_8\setminus D_1)\cap B'_n}|\nabla w|^2dx=\lim_{n\to\infty}\int_{(D_8\setminus D_1)\cap B'_n}|\nabla w_n|^2dx\geq \epsilon_6>0.
\end{equation}
By the conformal invariance of harmonic map and Theorem \ref{thm:remov-sing-boundary}, $w(x)$ is a nontrivial harmonic disk which can be seen as the second bubble. This  contradicts the  ``one bubble" assumption.

\

\noindent$\textbf{(b)}$ $w_n$ has another energy concentration point $p\neq 0$.

\

Without loss of generality, we may assume $p$ is the only energy concentration point in $D^+_{r_0}(p)$ for some $r_0>0$.
Similar to the process of constructing the first bubble, there exist $x_n'\to p$ and $r_n'\to 0$ such that
\begin{equation}\label{equation:20}
E(w_n;D^+_{r'_n}(x'_n)\cap B_n')=\sup_{\substack{x\in D_{r_0}^+(p),r\leq r_n\\D^+_r(x)\subset D_{r_0}^+(p)}}E(w_n;D^+_r(x)\cap B_n')=\frac{\overline{\epsilon}^2}{32}.
\end{equation}
By \eqref{equation:02}, we know $r_n't_n\geq r_n$. Then, passing to a subsequence we may assume $\lim_{n\to\infty}\frac{d_n}{r'_nt_n}=d\in [0,a]$. Moreover, there exists a nontrivial harmonic map $\widetilde{v}^2(x):\mathbb{R}^2_d\to N$ with free boundary $\widetilde{v}^2(\partial \mathbb{R}^2_d)$ on $K$ satisfying, for any $R>0$,
\[
\lim_{n\to\infty}\|w_n(x_n'+r_n'x)- \widetilde{v}^2(x)\|_{W^{1,2}(D_R(0)\cap B''_n)}=0.
\]
where $B''_n:=\{x\in\mathbb{R}^2|x'_n+r'_nx\in B'_n\}$. That is
\begin{align}\label{equation:24}
\lim_{n\to\infty}\|u_n(x_n+t_nx_n'+t_nr_n'x)- \widetilde{v}^2(x)\|_{W^{1,2}(D_R(0)\cap B''_n)}=0.
\end{align}
Therefore, $\widetilde{v}^2(x)$ is also a bubble for the sequence $u_n$. This is also contradiction to the "one bubble" assumption. Thus, we proved \textbf{Claim} \eqref{equation:assumption-small}.

\

Let $x_n'\in\partial^0 D^+$ be the projection of $x_n$, $i.e.$ $d_n=dist(x_n,\partial^0D^+)=|x_n-x_n'|$. Firstly, we decompose the neck domain $D^+_\delta(x_n)\setminus D^+_{r_nR}(x_n)$ as follows
\begin{align*}
D^+_\delta(x_n)\setminus D^+_{r_nR}(x_n)=&D^+_\delta(x_n)\setminus D^+_{\frac{\delta}{2}}(x'_n)\cup D^+_{\frac{\delta}{2}}(x'_n)\setminus D^+_{2r_nR}(x'_n)\cup D^+_{2r_nR}(x'_n)\setminus D^+_{r_nR}(x_n)\\
:=&\Omega_1\cup\Omega_2\cup\Omega_3,
\end{align*}
when $n$ and $R$ are large.

Since $\lim_{n\to\infty}\frac{d_n}{r_n}=a$, when $n$ and $R$ are large enough, it is easy to see that
\[
\Omega_1\subset D^+_\delta(x_n)\setminus D^+_{\frac{\delta}{4}}(x_n)\quad and \quad \Omega_3\subset D^+_{4r_nR}(x_n)\setminus D^+_{r_nR}(x_n).
\]
Moreover, for any $2r_nR\leq t\leq \frac{1}{2}\delta$, there holds
\[
D^+_{2t}(x'_n)\setminus D^+_{t}(x'_n)\subset D^+_{4t}(x_n)\setminus D^+_{t/2}(x_n).
\]

By assumption \eqref{equation:assumption-small}, we have
\begin{equation}\label{equation:03}
E(u_n;\Omega_1)+E(u_n;\Omega_3)\leq \epsilon^2
\end{equation}
and
\begin{equation}\label{equation:assumption-small-2}
\int_{D^+_{2t}(x'_n)\setminus D^+_{t}(x'_n)}|\nabla u_n|^2dx\leq \epsilon^2 \mbox{ for any } t\in(2r_nR, \frac{1}{2}\delta).
\end{equation}

By a scaling argument, we may assume
\[
\|\nabla u_n\|_{L^2(D^+_{4t}(x'_n)\setminus D^+_{t/2}(x'_n))}+\|\tau(u_n)\|_{L^p(D^+_{4t}(x'_n)\setminus D^+_{t/2}(x'_n))}\leq\overline{\epsilon}.
\]
According to the small energy regularity theory Lemma \ref{lem:small-energy-regularity-interior} and Lemma \ref{lem:small-energy-regularity}, we obtain
\begin{align}\label{inequality:14}
Osc_{D^+_{2t}(x'_n)\setminus D^+_{t}(x'_n)}u_n\leq C(\|\nabla u_n\|_{L^2(D^+_{4t}(x'_n)\setminus D^+_{t/2}(x'_n))}+t\|\tau(u_n)\|_{L^2(D^+_{4t}(x'_n)\setminus D^+_{t/2}(x'_n))})
\end{align}
for any $t\in(2r_nR, \frac{1}{2}\delta)$. Thus, $u_n(\Omega_2)\subset K_{\delta_0}$ and we can extend the definition of $u_n$ to the domain $\widehat{\Omega}_2:= D_{\frac{\delta}{2}}(x'_n)\setminus D_{2r_nR}(x'_n)$ by defining $\widehat{u}_n$ as \eqref{def:function}. Then $\widehat{u}_n\in W^{2,2}(\widehat{\Omega}_2)$ and satisfies equation \eqref{equation:global-form} where we take $F_n(x)=\tau(u_n)(x)$ and define $\Upsilon_{\widehat{u_n}}(\cdot,\cdot)$, $\widehat{F_n}(x)$ as in \eqref{def:2}.

Define
\[
\widehat{u}_n^*(r):=\frac{1}{2\pi r}\int_{\partial D_r(x_n')}\widehat{u}_n.
\]
Then by \eqref{inequality:14}, we have
\begin{align*}
\|\widehat{u}_n(x)-\widehat{u}_n^*(x)\|_{L^\infty(\widehat{\Omega}_2)}&\leq\sup_{2r_nR\leq t\leq \frac{\delta}{2}} \|\widehat{u}_n(x)-\widehat{u}_n^*(x)\|_{L^\infty(D_{2t}(x_n')\setminus D_t(x_n'))}\\
&\leq C(1+\|D\sigma\|_{L^\infty})Osc_{D^+_{2t}(x_n')\setminus D^+_t(x_n')}u_n
\leq C(N)( \epsilon+\delta).\end{align*}

We have
\begin{align*}
\int_{\widehat{\Omega}_2}\nabla{\widehat{u}_n}\nabla(\widehat{u}_n-\widehat{u}_n^*)dx=\int_{\partial\widehat{\Omega}_2}\frac{\partial\widehat{u}_n}{\partial n}(\widehat{u}_n-\widehat{u}_n^*)-\int_{\widehat{\Omega}_2}\Delta{\widehat{u}_n}(\widehat{u}_n-\widehat{u}_n^*)dx.
\end{align*}

On the one hand, by Jessen's inequality, we have
\begin{align*}
\int_{\widehat{\Omega}_2}\nabla{\widehat{u}_n}\nabla(\widehat{u}_n-\widehat{u}_n^*)dx
&=\int_{\widehat{\Omega}_2}|\nabla{\widehat{u}_n}|^2dx-\int_{\widehat{\Omega}_2}\frac{\partial\widehat{u}_n}{\partial r}\frac{\partial\widehat{u}_n^*}{\partial r}dx\\
&\geq\int_{\widehat{\Omega}_2}|\nabla{\widehat{u}_n}|^2dx-(\int_{\widehat{\Omega}_2}|\frac{\partial\widehat{u}_n}{\partial r}|^2dx)^{1/2}(\int_{\widehat{\Omega}_2}|\frac{1}{2\pi}\int_0^{2\pi}
\frac{\partial \widehat{u}_n}{\partial r}(r,\theta)d\theta|^2dx)^{1/2}\\
&\geq\int_{\widehat{\Omega}_2}|\nabla{\widehat{u}_n}|^2dx-\int_{\widehat{\Omega}_2}|\frac{\partial\widehat{u}_n}{\partial r}|^2dx\\
&=\frac{1}{2}\int_{\widehat{\Omega}_2}|\nabla{\widehat{u}_n}|^2dx-\int_{\widehat{\Omega}_2}(|\frac{\partial\widehat{u}_n}{\partial r}|^2-\frac{1}{2}|\nabla{\widehat{u}_n}|^2)dx.
\end{align*}

On the other hand, using equation \eqref{equation:global-form}, we get
\begin{align*}
-\int_{\widehat{\Omega}_2}\Delta{\widehat{u}_n}(\widehat{u}_n-\widehat{u}_n^*)dx
&\leq
\int_{\widehat{\Omega}_2}|\Upsilon_{\widehat{u}_n}(\nabla\widehat{u}_n,\nabla\widehat{u}_n)
+\widehat{F_n}||\widehat{u}_n-\widehat{u}_n^*|dx\\
&\leq
C( \epsilon+\delta)\int_{\widehat{\Omega}_2}|\nabla\widehat{u}_n|^2dx+
C( \epsilon+\delta)\int_{\widehat{\Omega}_2}|\widehat{F_n}|dx\\
&\leq
C( \epsilon+\delta)\int_{\widehat{\Omega}_2}|\nabla\widehat{u}_n|^2dx+
C( \epsilon+\delta)\|\tau_n\|_{L^2(\Omega_2)}.
\end{align*}
Thus,
\begin{align}\label{inequality:03}
&(\frac{1}{2}-C( \epsilon+\delta))\int_{\widehat{\Omega}_2}|\nabla\widehat{u}_n|^2dx\notag\\
&\leq
\int_{\partial (\widehat{\Omega}_2)}\frac{\partial\widehat{u}_n}{\partial n}(\widehat{u}_n-\widehat{u}_n^*)+\int_{\widehat{\Omega}_2}(|\frac{\partial\widehat{u}_n}{\partial r}|^2-\frac{1}{2}|\nabla{\widehat{u}_n}|^2)dx+C( \epsilon+\delta).
\end{align}

By the definition of $\widehat{u}_n$ (see \eqref{def:function}), we obtain
\begin{align*}
&\int_{\widehat{\Omega}_2}(|\frac{\partial\widehat{u}_n}{\partial r}|^2-\frac{1}{2}|\nabla{\widehat{u}_n}|^2)dx\\
&=
\int_{\Omega_2}(|\frac{\partial u_n}{\partial r}|^2-\frac{1}{2}|\nabla{u_n}|^2)dx+\int_{\widehat{\Omega}_2\setminus \Omega_2}(|D\sigma\cdot\frac{\partial u_n(\rho(x))}{\partial r}|^2-\frac{1}{2}|D\sigma\cdot\nabla{u_n(\rho(x))}|^2)dx\\
&=
\int_{\Omega_2}(|\frac{\partial u_n}{\partial r}|^2-\frac{1}{2}|\nabla{u_n}|^2)dx+\int_{\Omega_2}(|D\sigma\cdot\frac{\partial u_n(x)}{\partial r}|^2-\frac{1}{2}|D\sigma\cdot\nabla{u_n(x)}|^2)dx.
\end{align*}
Note that
\begin{align*}
|D\sigma\cdot\frac{\partial u_n(x)}{\partial r}|^2&=\langle P(u_n(x))\cdot\frac{\partial u_n(x)}{\partial r},P(u_n(x))\cdot\frac{\partial u_n(x)}{\partial r}\rangle
=\langle P^TP\cdot\frac{\partial u_n(x)}{\partial r},\frac{\partial u_n(x)}{\partial r}\rangle\\
&=\langle \big(P^TP-Id\big)\frac{\partial u_n(x)}{\partial r},\frac{\partial u_n(x)}{\partial r}\rangle+|\frac{\partial u_n(x)}{\partial r}|^2,
\end{align*}
where $P$ is the  matrix corresponding to the  linear operator defined by \eqref{def:1} under the orthonormal basis of $\mathbb{R}^N$.

Similarly,
\begin{align*}
|D\sigma\cdot\nabla u_n(x)|^2=
\langle \big(P^TP-Id\big)\nabla u_n(x),\nabla u_n(x)\rangle+|\nabla u_n(x)|^2.
\end{align*}

Noting that $\Xi|_{K}=Id$, by the continuity of eigenvalues of $P^TP$, we have that for any $\delta'>0$, there exists a constant $\delta_1=\delta_1(\delta')>0$, such that for any $\xi\in \mathbb{R}^n$ and $y\in K_{\delta_1}$, there holds $$\langle P^T(y)P(y)\xi,\xi\rangle\leq (1+\delta')|\xi|^2.$$

By \eqref{inequality:14}, we have $\|dist( u_n,K)\|_{L^\infty(\Omega_2)}\leq C(\epsilon+\delta)$. Thus, for any $\delta'>0$, $\xi\in \mathbb{R}^n$, there holds $$\langle(P^T(u_n(x))P(u_n(x))-Id)\xi,\xi\rangle\leq \delta'|\xi|^2$$ when $\epsilon$ and $\delta$ are small enough.

Thus,
\begin{align}\label{inequality:04}
&\int_{\widehat{\Omega}_2}(|\frac{\partial\widehat{u}_n}{\partial r}|^2-\frac{1}{2}|\nabla{\widehat{u}_n}|^2)dx\notag\\
&\leq
2\int_{\Omega_2}(|\frac{\partial u_n}{\partial r}|^2-\frac{1}{2}|\nabla{u_n}|^2)dx+C\delta'\int_{\Omega_2}|\nabla{u_n(x)}|^2dx\notag\\
&=
2\sum_{i=1}^{m_n}\int_{D^+_{2^{i}(2r_nR)}(x'_n)\setminus D^+_{2^{i-1}(2r_nR)}(x'_n)}(|\frac{\partial u_n}{\partial r}|^2-\frac{1}{2}|\nabla{u_n}|^2)dx+C\delta'\int_{\Omega_2}|\nabla{u_n(x)}|^2dx\notag\\
&\leq
C\sum_{i=1}^{m_n}2^{i}r_nR+C\delta'\int_{\Omega_2}|\nabla{u_n(x)}|^2dx\leq C\delta+C\delta'\int_{\Omega_2}|\nabla{u_n(x)}|^2dx,
\end{align}
where the last second inequality follows from Corollary \ref{cor:01}.

Combining inequality \eqref{inequality:03} with \eqref{inequality:04}, we have
\begin{align}\label{inequality:05}
&(\frac{1}{2}-C(\epsilon+\delta'+\delta))\int_{\widehat{\Omega}_2}|\nabla\widehat{u}_n|^2dx\leq
\int_{\partial \widehat{\Omega}_2}\frac{\partial\widehat{u}_n}{\partial n}(\widehat{u}_n-\widehat{u}_n^*)+C(\epsilon+\delta).
\end{align}

As for the boundary term, by trace theory, we have
\begin{align*}
\int_{\partial D_{\frac{1}{2}\delta}(x_n')}\frac{\partial\widehat{u}_n}{\partial n}(\widehat{u}_n-\widehat{u}_n^*)
&\leq C(\epsilon+\delta)\int_{\partial^+ D_{\frac{1}{2}\delta}(x_n')}|\nabla u_n|\\
&\leq
C(\epsilon+\delta)\left(\|\nabla u_n\|_{L^2(D^+_{\frac{1}{2}\delta}(x_n')\setminus D^+_{\frac{1}{4}\delta}(x_n') )}+\delta\|\nabla^2 u_n\|_{L^2(D^+_{\frac{1}{2}\delta}(x_n')\setminus D^+_{\frac{1}{4}\delta}(x_n') )}\right)\\
&\leq
C(\epsilon+\delta)\left(\|\nabla u_n\|_{L^2(D^+_{\delta}(x_n')\setminus D^+_{\frac{1}{8}\delta}(x_n') )}+\delta\|\tau_n\|_{L^2(D^+_{\delta}(x_n')\setminus D^+_{\frac{1}{8}\delta}(x_n') )}\right)\\
&\leq C(\epsilon+\delta),
\end{align*}
where the last second inequality can be derived from Lemma \ref{lem:small-energy-regularity-interior} and Lemma \ref{lem:small-energy-regularity}.

Also, there holds
\begin{align*}
\int_{\partial D_{2r_nR}(x_n')}\frac{\partial\widehat{u}_n}{\partial n}(\widehat{u}_n-\widehat{u}_n^*)
\leq C(\epsilon+\delta).
\end{align*}

Therefore, combining these results and taking $\epsilon$ and $\delta$ in \eqref{inequality:05} sufficiently small (then $\delta'$ is small), we have
\begin{align}\label{inequality:06}
\int_{\Omega_2}|\nabla u_n|^2dx\leq
\int_{\widehat{\Omega}_2}|\nabla\widehat{u}_n|^2dx\leq C(\delta+\epsilon).
\end{align}
Then the equality \eqref{equation:energy-identity} follows from \eqref{equation:03} and \eqref{inequality:06}. We proved the energy identity for the \textbf{Case 1}.

\

\noindent\textbf{Step 2:} We prove the energy identity for \textbf{Case 2}, $i.e.$, $\limsup_{n\to\infty}\frac{d_n}{r_n}=\infty$.

\

The proof is similar to the \textbf{Case 1}. Firstly, we need to show the \textbf{Claim} \eqref{equation:assumption-small} also holds in this case.

In fact, if \eqref{equation:assumption-small} is not true, then we can find $t_n\to 0$, such that $\lim_{n\to \infty}\frac{t_n}{r_n}=\infty$ and
\begin{equation}\label{equation:21}
\int_{D^+_{8t_n}(x_n)\setminus D^+_{t_n}(x_n)}|\nabla u_n|^2dx\geq \epsilon_6>0.
\end{equation}
 Then passing to a subsequence, we may assume
\[
\lim_{n\to\infty}\frac{d_n}{t_n}=b\in[0,\infty].
\]

We set $$w_n(x):=u_n(x_n+t_nx)$$ and $$B'_n:=\{x\in\mathbb{R}^2|x_n+t_nx\in D^+\}.$$ Then $w_n(x)$ lives in $B'_n$ and $0$ is an energy concentration point for $w_n$. We have to consider the following two cases:

\

\noindent$\textbf{(c)}$ $b<\infty$.

\

Then $B'_n$ tends to $\mathbb{R}^2_b$ as $n\to \infty$. Here, we also need to consider two cases.

\

\noindent$\textbf{(i)}$ $w_n$ has no other energy concentration points except $0$. It is almost the same as $\textbf{Case (a)}$ in \textbf{Step 1} where by passing to a subsequence, $w_n$ converges to a nontrivial harmonic map $w(x):\mathbb{R}^2_b\to N$ with free boundary $w(\partial \mathbb{R}^2_b)$ on $K$ which can be seen as the second bubble. This is a contradiction  to the "one bubble" assumption.

\

\noindent$\textbf{(ii)}$ $w_n$ has another energy concentration point $p\neq 0$. Similar to the process of $\textbf{Case (b)}$ in \textbf{Step 1}, there exist $x_n'\to p$ and $r_n'\to 0$ such that \eqref{equation:20} holds. Then, passing to a subsequence, we may assume $$\lim_{n\to\infty}\frac{d_n}{r'_nt_n}=d\in [0,\infty].$$ Moreover, if $d\in [0,\infty)$, then there exists a nontrivial harmonic map $\widetilde{v}^2(x):\mathbb{R}^2_d\to N$ with free boundary $\widetilde{v}^2(\partial \mathbb{R}^2_d)$ on $K$ satisfying \eqref{equation:24} as in $\textbf{Case (b)}$. If $d=\infty$, by the process of constructing the first bubble in \textbf{Case 2}, there exists $v^2(x):\mathbb{R}^2\to N$ is a nontrivial harmonic map such that $$w_n(x_n'+r_n'x)\to v^2(x)\ in\ W^{1,2}_{loc}(\mathbb{R}^2),$$ that is $$u_n(x_n+t_nx_n'+t_nr_n'x)\to v^2(x)\ in\ W^{1,2}_{loc}(\mathbb{R}^2).$$ In both cases, we will get the second bubble $v^2(x)$ or $\widetilde{v}^2(x)$. This contradicts the "one bubble" assumption.

\

\noindent$\textbf{(d)}$ $b=\infty$.

\

Then $B'_n$ tends to $\mathbb{R}^2$ as $n\to \infty$. Again, we need to consider two cases.

\

\noindent$\textbf{(iii)}$ $w_n$ has no other energy concentration points except $0$. By Lemma \ref{lem:small-energy-regularity-interior}, Theorem \ref{thm:remov-sing-interior} and \eqref{equation:21}, there exists $v^2(x):\mathbb{R}^2\to N$ is a nontrivial harmonic map such that $$w_n(x)\to v^2(x)\ in\ W^{1,2}_{loc}(\mathbb{R}^2\setminus \{0\}).$$ Then, we get the second bubble $v^2(x)$ which  contradicts the "one bubble" assumption.

\

\noindent$\textbf{(iv)}$ $w_n$ has another energy concentration point $p\neq 0$. Similar to  $\textbf{Case (b)}$ in \textbf{Step 1}, there exist $x_n'\to p$ and $r_n'\to 0$ such that \eqref{equation:20} holds and passing to a subsequence, we have $$\lim_{n\to\infty}\frac{d_n}{r'_nt_n}=\infty.$$ Moreover, by the process of constructing the first bubble in \textbf{Case 2}, there exists a nontrivial harmonic map $v^2(x):\mathbb{R}^2\to N$  such that $$w_n(x_n'+r_n'x)\to v^2(x)\ in\ W^{1,2}_{loc}(\mathbb{R}^2),$$ that is $$u_n(x_n+t_nx_n'+t_nr_n'x)\to v^2(x)\ in\ W^{1,2}_{loc}(\mathbb{R}^2).$$ This is also a contradiction to the  "one bubble" assumption. Thus, we proved our \textbf{Claim} \eqref{equation:assumption-small}.

\

Secondly, we decompose the neck domain $D^+_\delta(x_n)\setminus D^+_{r_nR}(x_n)$ as follows
\begin{align*}
D^+_\delta(x_n)\setminus D^+_{r_nR}(x_n)&=D^+_\delta(x_n)\setminus D^+_{\frac{\delta}{2}}(x'_n)\cup D^+_{\frac{\delta}{2}}(x'_n)\setminus D^+_{2d_n}(x'_n)\\
&\quad\cup D^+_{2d_n}(x'_n)\setminus D^+_{d_n}(x_n)\cup D^+_{d_n}(x_n)\setminus D^+_{r_nR}(x_n)\\
&:=\Omega_1\cup\Omega_2\cup\Omega_3\cup\Omega_4,
\end{align*}
when $n$ is large.

Since $\lim_{n\to\infty}d_n=0$ and $\lim_{n\to\infty}\frac{d_n}{r_n}=\infty$, when $n$ is large enough, it is easy to see that
\[
\Omega_1\subset D^+_\delta(x_n)\setminus D^+_{\frac{\delta}{4}}(x_n),\quad and \quad\Omega_3\subset D^+_{4d_n}(x_n)\setminus D^+_{d_n}(x_n).
\]

Moreover, for any $2d_n\leq t\leq \frac{1}{2}\delta$, there holds
\[
D^+_{2t}(x'_n)\setminus D^+_{t}(x'_n)\subset D^+_{4t}(x_n)\setminus D^+_{t/2}(x_n).
\]

By assumption \eqref{equation:assumption-small}, we have
\begin{equation}\label{inequality:07}
E(u_n;\Omega_1)+E(u_n;\Omega_3)\leq \epsilon^2
\end{equation}
and
\begin{equation}
\int_{D^+_{2t}(x'_n)\setminus D^+_{t}(x'_n)}|\nabla u_n|^2dx\leq \epsilon^2 \mbox{ for any } t\in(2d_n, \frac{1}{2}\delta).
\end{equation}

Noting that $\Omega_4=D^+_{d_n}(x_n)\setminus D^+_{r_nR}(x_n)=D_{d_n}(x_n)\setminus D_{r_nR}(x_n)$, by the well-known blow-up analysis theory of harmonic maps with interior blow-up points (also a sequence of maps with uniformly $L^p$ bounded tension fields for some $p\geq\frac{6}{5}$), there holds
\begin{equation}\label{equation:04}
\lim_{R\to\infty}\lim_{n\to 0}E(u_n;D_{d_n}(x_n)\setminus D_{r_nR}(x_n))=0.
\end{equation}
and
\begin{equation}
\lim_{R\to\infty}\lim_{n\to 0}Osc(u_n)_{D_{d_n}(x_n)\setminus D_{r_nR}(x_n)}=0.
\end{equation}
See \cite{DingWeiyueandTiangang,LiJiayuandZhuXiangrong,qing1997bubbling} for details.

Lastly, to estimate the energy concentration in $\Omega_2$, we can use the  same argument as in the  previous \textbf{Case 1} to get
\begin{align}\label{inequality:08}
\int_{\Omega_2}|\nabla u_n|^2dx\leq C(\delta+\epsilon).
\end{align}
Combining \eqref{inequality:07}, \eqref{equation:04} with \eqref{inequality:08}, it is easy to obtain \eqref{equation:energy-identity}. We proved the energy identity.

\

Next, we prove the no neck   property in Theorem \ref{thm:1}, $i.e.$, the base map and the bubbles are connected in the target manifold.

\

\noindent\textbf{No neck property}:
Here, we also need to consider two cases. But, for \textbf{Case 2}, we use the same argument as in the  previous reasoning where we split the neck domain into two parts, an interior domain and a boundary domain. Then, with the help of the no neck results in \cite{qing1997bubbling,LiJiayuandZhuXiangrong} for a sequence of maps with uniformly $L^2$-bounded tension fields, we just need to prove \eqref{equation:no-neck} for \textbf{Case 1}.

We may assume $\lim_{n\to\infty}\frac{d_n}{r_n}=a$ and decompose the neck domain $D^+_\delta(x_n)\setminus D^+_{r_nR}(x_n)=\Omega_1\cup\Omega_2\cup\Omega_3$, when $n$ and $R$ are large.

By assumption \eqref{equation:assumption-small} and small energy regularity (Lemma \ref{lem:small-energy-regularity-interior} and Lemma \ref{lem:small-energy-regularity}), we have
\begin{align}\label{inequality:09}
\|u_n\|_{Osc(D^+_{\delta}(x_n)\setminus D^+_{\frac{\delta}{4}}(x'_n))}&\leq \|u_n\|_{Osc(D^+_{\delta}(x_n)\setminus D^+_{\frac{\delta}{5}}(x_n))}\notag\\
&\leq C(\|\nabla u_n\|_{L^2(D^+_{\frac{4\delta}{3}}(x_n)\setminus D^+_{\frac{\delta}{6}}(x_n))}+\delta\|\tau_n\|_{L^2(D^+_{\frac{4\delta}{3}}(x_n)\setminus D^+_{\frac{\delta}{6}}(x_n))})\leq C(\epsilon+\delta)
\end{align}
and
\begin{align}\label{inequality:10}
\|u_n\|_{Osc(D^+_{4r_nR}(x'_n)\setminus D^+_{r_nR}(x_n))}&\leq \|u_n\|_{Osc(D^+_{5r_nR}(x_n)\setminus D^+_{r_nR}(x_n))}
\notag\\&\leq C(\|\nabla u_n\|_{L^2(D^+_{6r_nR}(x_n)\setminus D^+_{\frac{3r_nR}{4}}(x_n))}+r_nR\|\tau_n\|_{L^2(D^+_{6r_nR}(x_n)\setminus D^+_{\frac{3r_nR}{4}}(x_n))})\notag\\
&\leq C(\epsilon+\delta),
\end{align}
when $n$, $R$ are large and $\delta$ is small.

Without loss of generality, we may assume $\frac{1}{2}\delta=2^{m_n}(2r_nR)$ where $m_n\to \infty$ as $n\to\infty$.  Inspired by a technique by Ding \cite{DingWeiyue} for the interior bubbling case, we set $Q(t):=D^+_{2^{t_0+t}2r_nR}(x_n')\setminus D^+_{2^{t_0-t}2r_nR}(x_n')$, $\widehat{Q}(t):=D_{2^{t_0+t}2r_nR}(x_n')\setminus D_{2^{t_0-t}2r_nR}(x_n')$ and define
\[
f(t):=\int_{Q(t)}|\nabla u_n|^2dx,
\]
where $0\leq t_0\leq m_n$ and $0\leq t\leq \min\{t_0,m_n-t_0\}$.

Similar to the proof of \eqref{inequality:03} and \eqref{inequality:04}, we have
\begin{align}
&(\frac{1}{2}-C(\epsilon+\delta))\int_{\widehat{Q}(t)}|\nabla\widehat{u}_n|^2dx\notag\\&\leq
\int_{\partial (\widehat{Q}(t))}\frac{\partial\widehat{u}_n}{\partial n}(\widehat{u}_n-\widehat{u}_n^*)+\int_{\widehat{Q}(t)}(|\frac{\partial\widehat{u}_n}{\partial r}|^2-\frac{1}{2}|\nabla{\widehat{u}_n}|^2)dx+C(\epsilon+\delta)\int_{Q(t)}|\tau_n|dx
\end{align}
and
\begin{align}
\int_{\widehat{Q}(t)}(|\frac{\partial\widehat{u}_n}{\partial r}|^2-\frac{1}{2}|\nabla{\widehat{u}_n}|^2)dx&\leq
2\int_{Q(t)}(|\frac{\partial u_n}{\partial r}|^2-\frac{1}{2}|\nabla{u_n}|^2)dx+C\delta' \int_{Q(t)}|\nabla{u_n(x)}|^2dx\notag\\
&\leq
C2^{t_0+t}r_nR+C\delta'\int_{Q(t)}|\nabla{u_n(x)}|^2dx
\end{align}
where the last inequality follows from Corollary \ref{cor:01}.

As for the boundary, by Poincar\'{e}'s inequality, we have
\begin{align*}
\int_{\partial (D_{2^{t_0+t}2r_nR}(x_n'))}\frac{\partial\widehat{u}_n}{\partial n}(\widehat{u}_n-\widehat{u}_n^*)&\leq
(\int_{\partial (D_{2^{t_0+t}2r_nR}(x_n'))}|\frac{\partial\widehat{u}_n}{\partial r}|^2)^{\frac{1}{2}}(\int_{\partial (D_{2^{t_0+t}2r_nR}(x_n'))}|\widehat{u}_n-\widehat{u}_n^*|^2)^{\frac{1}{2}}\\
&\leq
C(\int_{\partial (D_{2^{t_0+t}2r_nR}(x_n'))}|\frac{\partial\widehat{u}_n}{\partial r}|^2)^{\frac{1}{2}}(2^{t_0+t}2r_nR\int_0^{2\pi}|\frac{\partial\widehat{u}_n}{\partial \theta}|^2)^{\frac{1}{2}}\\
&\leq
C2^{t_0+t}2r_nR\int_{\partial (D_{2^{t_0+t}2r_nR}(x_n'))}|\nabla\widehat{u}_n|^2\\
&\leq
C2^{t_0+t}2r_nR\int_{\partial^+ (D^+_{2^{t_0+t}2r_nR}(x_n'))}|\nabla u_n|^2.
\end{align*}
Similarly, we get
\begin{align*}
\int_{\partial (D_{2^{t_0-t}2r_nR}(x_n'))}\frac{\partial\widehat{u}_n}{\partial n}(\widehat{u}_n-\widehat{u}_n^*)\leq
C2^{t_0-t}2r_nR\int_{\partial^+ (D^+_{2^{t_0-t}2r_nR}(x_n'))}|\nabla u_n|^2.
\end{align*}

Using these together, we have
\begin{align}
&(\frac{1}{2}-C(\epsilon+\delta'+\delta))\int_{\widehat{Q}(t)}|\nabla\widehat{u}_n|^2dx\notag\\&\leq
C2^{t_0+t}2r_nR\int_{\partial^+ (D^+_{2^{t_0+t}2r_nR}(x_n'))}|\nabla u_n|^2+C2^{t_0-t}2r_nR\int_{\partial^+ (D^+_{2^{t_0-t}2r_nR}(x_n'))}|\nabla u_n|^2\notag\\
&\quad+C2^{t_0+t}r_nR+C(\epsilon+\delta)\int_{Q(t)}|\tau_n|dx\notag.
\end{align}
Taking $\epsilon$ and $\delta$ sufficiently small, we get
\begin{align}
&\int_{Q(t)}|\nabla u_n|^2dx\leq
C2^{t_0+t}2r_nR\int_{\partial^+ (D^+_{2^{t_0+t}2r_nR}(x_n'))}|\nabla u_n|^2+C2^{t_0-t}2r_nR\int_{\partial^+ (D^+_{2^{t_0-t}2r_nR}(x_n'))}|\nabla u_n|^2+C2^{t_0+t}r_nR.\notag
\end{align}
Therefore,
\begin{equation}
f(t)\leq \frac{C}{\log 2}f'(t)+C2^{t_0+t}r_nR.
\end{equation}
Thus,
\[
(2^{-\frac{1}{C}t}f(t))'\geq -C2^{t_0+(1-1/C)t}r_nR.
\]
Integrating from $2$ to $L$, we arrive at
\begin{align*}
f(2)&\leq C2^{-\frac{1}{C}L}f(L)+C2^{t_0}r_nR\int_2^L2^{(1-1/C)t}dt\leq C2^{-\frac{1}{C}L}f(L)+C2^{t_0}r_nR2^{(1-1/C)L}.
\end{align*}

Now, let $t_0=i$ and $L=L_i:=\min\{i,m_n-i\}$. Then, we have $Q(L_i)\subset D^+_{\delta/2}(x'_n)\setminus D^+_{2r_nR}(x'_n)\subset D^+_\delta(x_n)\setminus D^+_{r_nR}(x_n)$ and
\begin{align*}
\int_{D^+_{2^{i+2}2r_nR}(x_n')\setminus D^+_{2^{i-2}2r_nR}(x_n')}|\nabla u_n|^2dx
&\leq CE(u_n,D^+_\delta(x_n)\setminus D^+_{r_nR}(x_n))2^{-\frac{1}{C}L_i}+C2^{i}r_nR2^{(1-1/C)L_i}\\
&\leq
CE(u_n,D^+_\delta(x_n)\setminus D^+_{r_nR}(x_n))2^{-\frac{1}{C}L_i}+C2^{i}r_nR2^{(1-1/C)(m_n-i)}\\
&\leq
CE(u_n,D^+_\delta(x_n)\setminus D^+_{r_nR}(x_n))2^{-\frac{1}{C}L_i}+C\delta2^{(-1/C)(m_n-i)}\\
&\leq
C\epsilon2^{-\frac{1}{C}L_i}+C\delta2^{(-1/C)(m_n-i)},
\end{align*}
where we used the energy identity \eqref{equation:energy-identity}.

By Lemma \ref{lem:small-energy-regularity-interior} and Lemma \ref{lem:small-energy-regularity}, we obtain
\begin{align*}
&Osc_{D^+_{2^{i+1}2r_nR}(x_n')\setminus D^+_{2^{i-1}2r_nR}(x_n')}u_n\\
&\leq
C\left(\|\nabla u_n\|_{L^2(D^+_{2^{i+2}2r_nR}(x_n')\setminus D^+_{2^{i-2}2r_nR}(x_n'))}+(2^{i+2}2r_nR)\|\tau_n\|_{L^2(D^+_{2^{i+2}2r_nR}(x_n')\setminus D^+_{2^{i-2}2r_nR}(x_n'))}\right)\\
&\leq
C\left(\|\nabla u_n\|_{L^2(D^+_{2^{i+2}2r_nR}(x_n')\setminus D^+_{2^{i-2}2r_nR}(x_n'))}+2^{i}r_nR\right).
\end{align*}

Summing over $i$ from $2$ to $m_n-2$, we have
\begin{align*}
\|u_n\|_{Osc(D^+_{\delta/4}(x'_n)\setminus D^+_{4r_nR}(x'_n))}&\leq \sum_{i=2}^{m_n-2} \|u_n\|_{Osc(D^+_{2^{i+1}2r_nR}(x_n')\setminus D^+_{2^{i-1}2r_nR}(x_n'))}\\
&\leq
C\sum_{i=2}^{m_n-2}\left(\epsilon 2^{-\frac{1}{C}L_i}+\delta2^{(-1/C)(m_n-i)}+2^{i}r_nR\right)\\
&\leq
C\sum_{i=2}^{m_n-2}2^{-\frac{1}{C}i}(\epsilon +\delta)+C\delta\leq C(\epsilon+\delta).
\end{align*}
This inequality and \eqref{inequality:09}, \eqref{inequality:10} imply \eqref{equation:no-neck} and we have proved there is no neck during the blow-up process.
\end{proof}

\

Now, we can prove Theorem \ref{main:thm-1}.

\begin{proof}[\textbf{Proof of Theorem \ref{main:thm-1}}]
Combining the blow-up theory of a sequence of maps with uniformly $L^2$-bounded tension fields from a closed Riemann surface
(see \cite{DingWeiyueandTiangang,LiJiayuandZhuXiangrong,Lin-Wang,qing1997bubbling,Luoyong}) and Theorem \ref{thm:1}, we can easily
get the conclusion of Theorem \ref{main:thm-1} by following the standard blow-up scheme in \cite{DingWeiyueandTiangang}.

On the other hand, it is well known that harmonic spheres are minimal spheres and harmonic disks with free boundary on $K$ are minimal disks with free boundary on $K$ (see e.g. the proof of Theorem B in \cite{Ma-li}, page 300).
\end{proof}

\

\section{Application to the harmonic map flow with free boundary}\label{set:heat-flow}

\

In this section, we will apply the results in Theorem \ref{main:thm-1} to the harmonic map flow with free boundary and prove Theorem \ref{thm:infty-time} and Theorem \ref{thm:finite-time}.

Firstly, we have
\begin{lem}\label{lem:monotonicity}
Let $u:M\times(0,\infty)\to N$ be a global weak solution to (\ref{HMF:1}-\ref{HMF:4}), which is smooth away from a finite number of singular points. There holds the estimate
\begin{equation}
\int_0^\infty\int_M|\partial_tu|^2dxdt\leq E(u_0).
\end{equation}
Moreover, $E(u(\cdot,t))$ is continuous on $[0,\infty)$ and non-increasing.
\end{lem}
\begin{proof}
The proof is similar to Lemma 3.4 in \cite{Struwe-1}.
Multiply the equation \eqref{HMF:1} by $\partial_t u$ and integrate by parts, for any $0\leq t_1\leq t_2\leq\infty$, to get
\begin{align*}
\int_{t_1}^{t_2}\int_M|\partial_tu|^2dxdt&=\int_{t_1}^{t_2}\int_M-\Delta_g u\cdot\partial_tudxdt\\
&=\int_{t_1}^{t_2}\int_{\partial M}\frac{\partial u}{\partial\overrightarrow{n}}\cdot\partial_tu-\int_{t_1}^{t_2}\int_M\nabla u\cdot\nabla(\partial_tu)dxdt\\
&=-\int_{t_1}^{t_2}\int_M\frac{1}{2}\partial_t|\nabla u|^2dxdt=E(u(\cdot,t_1))-E(u(\cdot,t_2)),
\end{align*}
where $\overrightarrow{n}$ is the outward unit normal vector field on $\partial M$ and
we used the free boundary condition that $\frac{\partial u}{\partial\overrightarrow{n}}\bot\partial_tu$. Then the conclusion of the lemma follows immediately.
\end{proof}

Similar to the case of a closed domain (see Lemma 2.5 in \cite{Lin-Wang}), we have
\begin{lem}\label{lem:two-balls}
Let $u\in C^\infty(M\times (0,T_0),N)$ be a solution to (\ref{HMF:1}-\ref{HMF:4}). Then there exists a constant $R_0>0$ such that, for any
$x_0\in M$, $0<t\leq s<T_0$ and $0<R\leq R_0$, there hold:
\begin{align}\label{inequality:11}
E(u(s);B^M_{R}(x_0))\leq E(u(t);B^M_{2R}(x_0))+C\frac{s-t}{R^2}E(u_0),
\end{align}
and
\begin{align}\label{inequality:12}
E(u(t);B^M_{R}(x_0))\leq E(u(s);B^M_{2R}(x_0))+C\int_t^s\int_M|\partial_tu|^2dxdt+C\frac{s-t}{R^2}E(u_0).
\end{align}
\end{lem}
\begin{proof}
Let $\eta\in C^\infty_0(B^M_{2R}(x_0))$ be such that $0\leq\eta\leq 1$, $\eta|_{B^M_{R}(x_0)}\equiv 1$ and $|\nabla\eta|\leq\frac{C}{R}$. Multiplying \eqref{HMF:1} by $\eta^2\partial_t u$ and integrating by parts, we get
\begin{align*}
\int_M|\partial_tu|^2\eta^2dx+\frac{d}{dt}(\frac{1}{2}\int_M|\nabla u|^2\eta^2dx)&=\int_{\partial M}\frac{\partial u}{\partial \overrightarrow{n}}\cdot\partial_t u\eta^2-2\int_M\partial_tu\nabla u\eta\nabla\eta dx\\
&=-2\int_M\partial_tu\nabla u\eta\nabla\eta dx,
\end{align*}
where we used the free boundary condition that $\frac{\partial u}{\partial\overrightarrow{n}}\bot\partial_tu$.

Since
\[
|2\int_M\partial_tu\nabla u\eta\nabla\eta dx|\leq \frac{1}{2}\int_M|\partial_tu|^2\eta^2dx+2\int_M|\nabla u|^2|\nabla\eta|^2dx,
\]
we have
\begin{align*}
-\frac{3}{2}\int_M|\partial_tu|^2\eta^2dx-2\int_M|\nabla u|^2|\nabla\eta|^2dx\leq\frac{d}{dt}(\frac{1}{2}\int_M|\nabla u|^2\eta^2dx)\leq 2\int_M|\nabla u|^2|\nabla\eta|^2dx.
\end{align*}
Integrating the above inequality from $t$ to $s$, we will get the conclusion of the lemma.
\end{proof}

With the help of Lemma \ref{lem:two-balls}, we can apply the standard argument for the closed case (see Lemma 4.1 in \cite{Lin-Wang}) to obtain the following:
\begin{lem}\label{lem:2}
Let $u\in C^\infty(M\times (0,T_0),N)$ be a solution to (\ref{HMF:1}-\ref{HMF:4}). Suppose $x_0\in M$ is the only singular point at time $T_0$.
Then there exists a positive number $m>0$ such that
\begin{align}\label{equation:05}
|\nabla u|^2(x,t)dx\to m\delta_{x_0}+|\nabla u|^2(x,T_0)dx,
\end{align}
for $t \uparrow T_0$, as Radon measures. Here, $\delta_{x_0}$ denotes the $\delta-$mass at $x_0$.
\end{lem}

Now, we begin to prove Theorem \ref{thm:infty-time} and Theorem \ref{thm:finite-time}.
Firstly, it is easy to see that Lemma \ref{lem:monotonicity}, Lemma \ref{lem:2} and Theorem \ref{main:thm-1} imply Theorem \ref{thm:infty-time}. In fact,

\begin{proof}[\textbf{Proof of Theorem \ref{thm:infty-time}}]
By Lemma \ref{lem:monotonicity}, we can find a sequence $t_n\uparrow\infty$ such that
\begin{align*}
\lim_{n\to\infty}\int_M|\partial_t u|^2(\cdot,t_n)dx=0\quad and \quad E(u(\cdot,t_n))\leq E(u_0).
\end{align*}
Take the sequence $u_n=u(\cdot,t_n)$, $\tau(u_n)=\partial_tu(\cdot,t_n)$ in Theorem \ref{main:thm-1}. Combining this with Lemma \ref{lem:2},  the conclusion of Theorem \ref{thm:infty-time} follows immediately.
\end{proof}

\begin{proof}[\textbf{Proof of Theorem \ref{thm:finite-time}}]
It is sufficient to consider the case that $(x_0,T_0)$ with $x_0\in\partial M$ being the only singular point at time $T_0$.
For the case of an interior singularity $x_0\in M\setminus\partial M$, one can refer to \cite{Lin-Wang}.
Without loss of generality, we may assume $M=D^+_1(0)$ and $x_0=0$. By Lemma \ref{lem:2}, there exist sequences $t_n\uparrow T_0$ and $\lambda_n\downarrow 0$ such that
\[
\lim_{n\to\infty}\int_{D^+_{\lambda_n}(0)}|\nabla u|^2(\cdot,t_n)dx=m.
\]

Let $u_n(x,t)=u(\lambda_nx,t_n+\lambda_n^2t).$ Without loss of generality, we may assume $t_n-2\lambda_n^2>0$. Then $u_n$ is defined in $D^+_{\lambda_n^{-1}}(0)\times [-2,0]$ satisfying \eqref{HMF:1} and
\[
\int_{-2}^0\int_{D^+_{\lambda_n^{-1}}(0)}|\partial_tu_n|^2dxdt=\int_{t_n-2\lambda_n^2}^{t_n}\int_{D^+_1(0)}|\partial_tu|^2dxdt\to 0
\]
as $n\to \infty$. By Fubini's theorem, there exists $s_n\in(-1,-\frac{1}{2})$ such that
\begin{equation}\label{equation:08}
\lim_{n\to\infty}\int_{D^+_{\lambda_n^{-1}}(0)}|\partial_tu_n|^2(\cdot,s_n)dx= 0.
\end{equation}

For the sequence $\{u_n(\cdot,s_n)\}$, there holds
\begin{equation}\label{equation:06}
\lim_{R\to\infty}\lim_{n\to\infty}\int_{D^+_R(0)}|\nabla u_n|^2(\cdot,s_n)dx=m.
\end{equation}
In fact, on the one hand, by \eqref{inequality:11}, we have
\begin{align*}
\int_{D^+_R(0)}|\nabla u_n|^2(\cdot,s_n)dx&=\int_{D^+_{\lambda_nR}(0)}|\nabla u|^2(\cdot,t_n+\lambda_n^2s_n)dx\geq
\int_{D^+_{\lambda_n}(0)}|\nabla u|^2(\cdot,t_n)dx-C\frac{1}{R^2}E(u_0).
\end{align*}
Thus,
\begin{equation}
\lim_{R\to\infty}\lim_{n\to\infty}\int_{D^+_R(0)}|\nabla u_n|^2(\cdot,s_n)dx\geq m.
\end{equation}

On the other hand, by \eqref{equation:05}, for any $R>0$ and $\sigma>0$, we have
\begin{align*}
\lim_{n\to\infty}\int_{D^+_{\lambda_nR}(0)}|\nabla u|^2(\cdot,t_n+\lambda_n^2s_n)dx
&\leq
\lim_{n\to\infty}\int_{D^+_{\sigma}(0)}|\nabla u|^2(\cdot,t_n+\lambda_n^2s_n)dx=m+\int_{D^+_{\sigma}(0)}|\nabla u|^2(\cdot,T_0)dx.
\end{align*}
Letting $\sigma\to 0$, we obtain
\begin{equation}\label{equation:07}
\lim_{n\to\infty}\int_{D^+_R(0)}|\nabla u_n|^2(\cdot,s_n)dx\leq m
\end{equation}
and \eqref{equation:06} follows immediately.

Fixing $R>0$, we consider the sequence $\{u_n(\cdot,s_n)\}_{n=1}^\infty$ which is defined in $D_R^+(0)$. By \eqref{equation:07} and \eqref{equation:08},
we know it is a sequence of maps from $D_R^+(0)$ to $N$ with finite energy and tension fields
$$\|\tau_n\|_{L^2(D^+_R(0))}=\|\partial_tu_n(\cdot,s_n)\|_{L^2(D^+_R(0))}\to 0$$
as $n\to \infty$. Moreover, for each $R>0$, $u_n(\cdot,s_n)$ weakly converges to a constant map. In fact, by Lemma \ref{lem:2}, for any $\sigma>0$, we have
\begin{align*}
\lim_{n\to\infty}E(u_n(\cdot,s_n),D_R^+\setminus D^+_\sigma)&=\lim_{n\to\infty}E(u(\cdot,t_n+\lambda_n^2 s_n),D_{\lambda_n R}^+\setminus D^+_{\lambda_n\sigma})\leq
\lim_{n\to\infty}E(u(\cdot,T_0),D_{\lambda_n R})=0.
\end{align*}

According to Theorem \ref{thm:1}, we know there exist $L_R$ nontrivial bubbles $\{w^i_R\}_{i=1}^{L_R}$ such that
\begin{equation}
\lim_{n\to\infty}E(u_n(\cdot,s_n),D_R^+)=\sum_{i=1}^{L_R}E(w_R^i).
\end{equation}

Since the energy of the bubble has a lower bound, $i.e.$ $E(w)\geq \overline{\epsilon_0}:=\min\{\epsilon_0,\epsilon_5\}$,
we have $1\leq L_R\leq \frac{m}{\overline{\epsilon_0}}+1$. Therefore, there exist a subsequence $R\uparrow\infty$ and a constant $L\in [1,\frac{m}{\overline{\epsilon_0}}+1]$ such that $L_R=L$ and
\begin{equation}
m=\lim_{R\to\infty}\lim_{n\to\infty}E(u_n(\cdot,s_n),D_R^+)
=\lim_{R\to\infty}\sum_{i=1}^{L}E(w_R^i).
\end{equation}

Using Theorem \ref{main:thm-1} with $M=S^2$ or $M=D$ and $\tau\equiv 0$, there exist $L_i$ bubbles $\{w^j\}_{j=1}^{L_i}$ such that
\begin{equation*}
\lim_{R\to\infty}E(w_R^i)=\sum_{j=1}^{L_i}E(w^j).
\end{equation*}
Then
\begin{equation}
m=\lim_{R\to\infty}\lim_{n\to\infty}E(u_n(\cdot,s_n),D_R^+)
=\lim_{R\to\infty}\sum_{i=1}^{L}E(w_R^i)=\sum_{i=1}^{L}\sum_{j=1}^{L_i}E(w^j).
\end{equation}
Combining with Lemma \ref{lem:2}, we obtain the conclusion of Theorem \ref{thm:finite-time}.
\end{proof}

\


\providecommand{\bysame}{\leavevmode\hbox to3em{\hrulefill}\thinspace}
\providecommand{\MR}{\relax\ifhmode\unskip\space\fi MR }
\providecommand{\MRhref}[2]{%
  \href{http://www.ams.org/mathscinet-getitem?mr=#1}{#2}
}
\providecommand{\href}[2]{#2}

\end{document}